\documentclass[11pt]{amsart}
\usepackage{amssymb, latexsym}
\theoremstyle{plain}
\newtheorem{theorem}{Theorem}

\newtheorem{proposition}{Proposition}

\newtheorem*{2'}{Theorem 2'}
\newtheorem*{3'}{Theorem 3'}

\theoremstyle{remark}

\newtheorem*{Remark 1}{Remark 1}
\newtheorem*{Remark 2}{Remark 2}
\newtheorem*{Remark 3}{Remark 3}
\newtheorem*{Remark 4}{Remark 4}

\numberwithin{equation}{section}

\begin{document}

\title [Optimizing  Drift in Diffusive Search for  Random  Target]
{Optimizing the drift in a diffusive search for a random stationary target}

\author{Ross G. Pinsky}


\address{Department of Mathematics\\
Technion---Israel Institute of Technology\\
Haifa, 32000\\ Israel}
\email{ pinsky@math.technion.ac.il}

\urladdr{http://www.math.technion.ac.il/~pinsky/}

\date{}

\begin{abstract}
Let $a\in\mathbb{R}$ denote an unknown stationary target  with a known distribution $\mu\in\mathcal{P(\mathbb{R}})$, the space of probability measures on $\mathbb{R}$. A  diffusive searcher $X(\cdot)$
sets out from the origin to locate the target. The time to locate
the target is $T_a=\inf\{t\ge0: X(t)=a\}$. The  searcher has a given constant diffusion rate $D>0$, but its drift $b$ can be set by the search designer
from a natural  admissible class $\mathcal{D}_\mu$ of drifts.
Thus, the diffusive searcher is a Markov process    generated by the operator $L=\frac D2\frac{d^2}{dx^2}+b(x)\frac d{dx}$.
For  a given drift $b$, the expected time of the search is
\begin{equation}
\int_{\mathbb{R}} (E^{(b)}_0T_a)\thinspace\mu(da).
\end{equation}
Our aim is to minimize this expected search time over all admissible drifts $b\in\mathcal{D}_\mu$.
For measures $\mu$ that satisfy a certain balance condition between their restriction to the positive axis and their restriction to the negative axis,
 a condition satisfied, in particular, by all symmetric measures, we can give a complete answer to the problem.
 We calculate the above infimum explicitly, we classify the measures for which the infimum is attained, and in the case that it is attained, we calculate the minimizing drift explicitly.
 For measures that do not satisfy the balance condition, we obtain partial results.

\end{abstract}

\maketitle
\section{Introduction and Statement of Results}\label{intro}
A number of recent papers have considered  a stochastic search  model for a stationary target  $a\in R^d$,
which might be random and have a known distribution attached to it, whereby a searcher sets off from a fixed point, say the origin, and performs  Brownian motion with diffusion constant $D$.
The searcher is also armed with   a (possibly space dependent) exponential resetting time, so that
if it has  failed to locate the target by that time, then it  begins its search anew from the origin.
One may be interested in several statistics, the most important one being  the expected time to locate the target. (In dimension one, the target is considered ``located'' when the process
hits the point $a$, while in dimensions two and higher, one chooses an $\epsilon>0$ and the target is said to be ``located''
when the process hits the $\epsilon$-ball centered at $a$.) Without the resetting, this expected time is infinite.
When the rate of the exponential clock is constant, the expected time to locate the target is finite; furthermore,
this jump-Brownian motion  process  possesses an  invariant probability density, call it $\nu$.
  See, for example, \cite{EM1, EM2,EMM}. For related models, see \cite{G,MV1,MV2} as well as the   references in all of the above  articles.

  It is well known that the Brownian motion with diffusion constant $D$ and with drift $\frac{D}2\frac{\nabla \nu}{\nu}$, that is the diffusion process generated by
$\frac D2\Delta+\frac{D}2\frac{\nabla \nu}{\nu}\cdot\nabla$, also has invariant probability density $\nu$.
In \cite{EMM}, for the case of constant  resetting rate in one dimension, it was shown that
the expected time to locate a target at the deterministic point $a\in\mathbb{R}$  for the jump-Brownian motion process  is less than the expected time for the
corresponding (non-jumping) diffusion process with the same invariant measure (generated by $\frac D2\frac{d^2}{dx^2}+\frac{D}2\frac{\nu'}{\nu}\frac{d}{dx}$)
to locate the target.
The above is partial motivation for the problem we consider in this paper; we believe it is also of some independent interest.

Let $a\in\mathbb{R}$ denote an unknown stationary target  with a known distribution $\mu\in\mathcal{P(\mathbb{R}})$, the space of probability measures on $\mathbb{R}$. A  diffusive searcher $X(\cdot)$ sets out from the origin to locate the target. The time to locate
the target is $T_a=\inf\{t\ge0: X(t)=a\}$. We assume that the diffusive searcher has a given constant diffusion rate $D>0$, but that its drift $b$ can be set by the search designer
from a natural  admissible class $\mathcal{D}_\mu$ of drifts, which we define below.
Thus, the  searcher is a Markov diffusion process    generated by the operator $L=\frac D2\frac{d^2}{dx^2}+b(x)\frac d{dx}$.
We will denote probabilities and expectations with respect to $X(\cdot)$ by $P_0^{(b)}$ and $E_0^{(b)}$.
For  a given drift $b$, the expected time of the search is
\begin{equation}
\int_{\mathbb{R}} (E^{(b)}_0T_a)\thinspace\mu(da).
\end{equation}
Our aim is to minimize this expected search time over all admissible drifts $b\in\mathcal{D}_\mu$.
We note that this same problem was recently considered in the physics literature \cite{K}; for more on this, see Remark 1 after Theorem \ref{thm2}.

We now discuss the influence of the  drift, which will lead us to the definition of the admissible class $\mathcal{D}_\mu$ of drifts.
In order to avoid trivialities, we will assume that the  support of $\mu$ has a non-empty intersection with both open half-lines.
(Otherwise, if say, $\mu$ is supported in $[0,\infty)$, then $\int_{\mathbb{R}} (E^{(b)}_0T_a)\thinspace\mu(da)$ is a decreasing function of the drift $b$ and converges to 0 as the drift converges pointwise to $+\infty$.)
For convenience only, we will assume that the origin is not an atom of the distribution $\mu$. We write $\mu$ in the form
\begin{equation}\label{mu}
\begin{aligned}
&\mu=(1-p)\mu_-+\thinspace p\mu_+,
\ \text{where}\ p\in(0,1), \ \mu_-\ \text{is a probability measure on}\ (-\infty,0)\\
& \text{and}\ \mu_+\ \text{is a probability measure on}\ (0,\infty).
\end{aligned}
\end{equation}
Define
\begin{equation}\label{support}
\begin{aligned}
&A_-(\mu)=\inf\{x\in(-\infty,0):\mu_-\big((-\infty,x]\big)>0\},\\
&  A_+(\mu)=\sup\{x\in(0,\infty):\mu_+\big([x,\infty)\big)>0\}.
\end{aligned}
\end{equation}
If $A_-(\mu)>-\infty$ ($A_+(\mu)<\infty$), then there is no point in searching to the left of $A_-(\mu)$ (to the right of $A_+(\mu)$), so we should  let $b(x)=+\infty$ for $x<A_-(\mu)$
($b(x)=-\infty$, for $x>A_+(\mu)$).
If $A_-(\mu)>-\infty$ ($A_+(\mu)<\infty$), then the diffusion can reach $A_-(\mu)$ ($A_+(\mu)$) if an only if $\int_{A_-}dx\exp(-\frac2D\int_0^xb(y)dy)<\infty$
($\int^{A_+}dx\exp(-\frac2D\int_0^xb(y)dy)<\infty$). (See \cite{P} or alternatively, the last paragraph of the proof of part (ii) of Theorem \ref{thm2} below.) If $A_-(\mu)>-\infty$ ($A_+(\mu)<\infty$) and the diffusion can reach $A_-(\mu)$ ($A_+(\mu)$), then setting $b(x)=+\infty$ for $x<A_-(\mu)$
($b(x)=-\infty$ for $x>A_+(\mu)$) is equivalent to imposing the reflecting boundary condition at $A_-(\mu)$ ($A_+(\mu)$).
 In terms of the generator $L$, the reflecting boundary condition at $A_-(\mu)$ (at $A_+(\mu)$) is equivalent to imposing the Neumann boundary condition $u'(A_-(\mu))=0$ ($u'(A_+(\mu))=0$).
 The above discussion  leads us to define the following condition on the  drift $b$\thinspace:
\begin{equation}\label{driftcond}
\begin{aligned}
 & b\ \text{\it is piecewise continuous and locally bounded on} \ (A_-(\mu),A_+(\mu)), \ \text{and}\\
 &\text{ \it is equal to}
 +\infty\ \text{\it on} \ (-\infty,A_-(\mu))\ \text{ \it and  to}\ -\infty\ \text{\it on}\ (A_+(\mu),\infty).\\
 &\text{\it Also, if}\ A_-(\mu)\ \text{\it is an atom for}\ \mu,\ \text{\it then}\ b\ \text{\it is locally bounded on}\  [A_-(\mu),A_+(\mu)),\\
& \text{\it and if}\ A_+(\mu)\ \text{\it is an atom for }\ \mu,\ \text{\it then}\ b\ \text{\it is locally bounded on}\ (A_-(\mu),A_+(\mu)].
\end{aligned}
\end{equation}
\noindent  \bf Remark.\rm\
In particular, if $\mu$ has atoms at both $A_-(\mu)$ and $A_+(\mu)$, then the drifts in \eqref{driftcond} are bounded
on $(A_-(\mu),A_+(\mu))$.


As is well-known, the expected hitting time $E^{(b)}_0T_a$  is finite for all
$a\in (A_-(\mu),A_+(\mu))$  if and only if
the diffusion $X(\cdot)$ is positive recurrent. Positive recurrence for drifts satisfying \eqref{driftcond} is equivalent to the condition
\begin{equation}\label{posrec}
\int_{A_-(\mu)}^{A_+(\mu)} dx\exp(\frac2D\int_0^xb(y)dy)<\infty.
\end{equation}
(See \cite{P}.)
We can now define the class of admissible drifts.

\noindent\bf The Class $\mathcal{D}_\mu$ of Admissible Drifts:\rm\
\begin{equation}\label{admiss}
 \begin{aligned}
& \mathcal{D}_\mu\   \text{is the class of drifts}\
 b\\
 &\text{satisfying \eqref{driftcond} and \eqref{posrec}}.
\end{aligned}
\end{equation}

Let
\begin{equation}\label{tails}
\overline{\mu}_-(x)=\mu_-((-\infty,x)),\ \text{for}\ x\le0, \ \ \overline{\mu}_+(x)=\mu_+((x,\infty)), \ \text{for}\ x\ge0,
\end{equation}
 denote the tails of $\mu_-$ and $\mu_+$.

We begin with the following result.
\begin{theorem}\label{thm1}
Let the target distribution $\mu$ satisfy $\mu=(1-p)\mu_-\thinspace +p\mu_+$  as in $\eqref{mu}$, let $A_-(\mu)$ and $A_+(\mu)$ be as in \eqref{support} and let
$\overline{\mu}_-(x)$ and $\overline{\mu}_+(x)$ be as in \eqref{tails}. Let the class of admissible drifts $\mathcal{D}_\mu$ be as in \eqref{admiss}.
If  $\int_{-\infty}^0\overline{\mu}_-^\frac12(x)dx=\infty$ and $\int_0^\infty\overline\mu_+^\frac12(x)dx=\infty$, then
\noindent $\int_{\mathbb{R}} (E^{(b)}_0T_a)\thinspace\mu(da)=\infty$, for all $b\in\mathcal{D}_\mu$.
\end{theorem}
\bf\noindent Remark.\rm\ Note of course that $\int_0^\infty\overline\mu_+^\frac12(x)dx=\int_0^{A_+(\mu)}\overline\mu_+^\frac12(x)dx$
and $\int_{-\infty}^0\overline{\mu}_-^\frac12(x)dx=\int_{A_-(\mu)}^0\overline{\mu}_-^\frac12(x)dx$, and that the first integral (second integral) is always
finite if $A_+(\mu)<\infty$  ($A_-(\mu)<\infty$).

The following simple proposition gives a sufficient moment condition for integrals of the above type to be finite.
\begin{proposition}\label{sufficientmoment}
Let $\nu$ be a probability measure on $(0,\infty)$ and let $\overline{\nu}(x)=\nu((x,\infty))$.
If $\int_0^\infty x^2|\log x|^{1+\epsilon}\nu(dx)<\infty$, for some $\epsilon>0$, then $\int_0^\infty\overline{\nu}^\frac12(x)dx<\infty$.
The condition $\int_0^\infty x^2|\log x|^{1-\epsilon}\nu(dx)<\infty$, for all $\epsilon\in(0,1)$,  is not sufficient for the finiteness of  $\int_0^\infty\overline{\nu}^\frac12(x)dx$.
\end{proposition}
\begin{proof}
For $\epsilon>0$,
$$
\int_0^\infty\overline{\nu}^\frac12(x)dx\le 2+ C\Big(\int_2^\infty x|\log x|^{1+\epsilon}\thinspace\overline{\nu}(x)\Big)^\frac12,
$$
where $C=\big(\int_2^\infty\frac1{x|\log x|^{1+\epsilon}}dx\big)^\frac12<\infty$.
An integration by parts shows that the integral on the right hand side above is finite if
$\int_0^\infty x^2(|\log x|)^{1+\epsilon}\nu(dx)<\infty$.
This proves the first claim in the proposition.
For the second claim, let
$\nu$ be a distribution that satisfies $\overline{\nu}(x)=\frac1{x^2(|\log x|)^2}$,    for $x\ge 2$.
Then  $\int_0^\infty x^2|\log x|^{1-\epsilon}\nu(dx)<\infty$, for all $\epsilon\in(0,1)$,  but  $\int_0^\infty\overline{\nu}^\frac12(x)dx=\infty$.
\end{proof}

In the case that $\int_{-\infty}^0\overline{\mu}_-^\frac12(x)dx$ and $\int_0^\infty\overline\mu_+^\frac12(x)dx$ are finite, the following condition
 on the target distribution $\mu$ will play a seminal role.
\medskip

\noindent \bf Square Root Balance Condition.\rm\
The target distribution $\mu=(1-p)\mu_-+p\mu_+$ is such that
the integrals $\int_0^\infty\overline\mu_+^\frac12(x)dx$ and $\int_{-\infty}^0\overline{\mu}_-^\frac12(x)dx$ are finite and satisfy
\begin{equation}\label{srbc}
\frac{\int_0^\infty\overline\mu_+^\frac12(x)dx}{\int_{-\infty}^0\overline{\mu}_-^\frac12(x)dx}=\frac{(1-p)\log(1-p)}{p\log p}.
\end{equation}
\noindent \bf Remark.\rm\ A symmetric target distribution (the case in which $\overline{\mu}_+(x)=\overline{\mu}_-(-x)$, for
$x\in(0,\infty)$, and $p=\frac12$) always satisfies the square root balance condition.

When the target distribution satisfies  the square root
balance condition, we can give a complete answer to the optimization problem.

\begin{theorem}\label{thm2}
Let the target distribution $\mu$ satisfy $\mu=(1-p)\mu_-\thinspace +p\mu_+$  as in $\eqref{mu}$, let $A_-(\mu)$ and $A_+(\mu)$ be as in \eqref{support} and let
$\overline{\mu}_-(x)$ and $\overline{\mu}_+(x)$ be as in \eqref{tails}. Let the class of admissible drifts $\mathcal{D}_\mu$ be as in \eqref{admiss}.
Assume also that the target distribution satisfies the square root balance condition \eqref{srbc}.
Then

\noindent i.
\begin{equation}\label{infb}
\begin{aligned}
&\inf_{b\in \mathcal{D}_\mu}\int_{\mathbb{R}} (E^{(b)}_0T_a)\thinspace\mu(da)=\\
&\frac2D\Big(\frac{1-p}{|\log p| }\thinspace(\int_{-\infty}^0\overline{\mu}_-^\frac12(x)dx)^2+\frac p{|\log(1-p)|}\thinspace(\int_0^\infty\overline{\mu}_+^\frac12(x)dx)^2\Big).
\end{aligned}
\end{equation}
In particular, in the case of a symmetric target distribution,
\begin{equation}\label{symmetriccaseinf}
\inf_{b\in \mathcal{D}_\mu}\int_{\mathbb{R}} (E^{(b)}_0T_a)\thinspace\mu(da)=\frac2{D\log 2}\thinspace(\int_0^\infty\overline{\mu}_+^\frac12(x)dx)^2.
\end{equation}

\noindent ii. The infimum in (i) is attained if and only if the restriction of   $\mu$  to $(A_-(\mu),A_+(\mu))$   is absolutely continuous with a piecewise continuous, locally bounded density.
($\mu$ may possess an atom at  $A_-(\mu)$ and/or at  $A_+(\mu)$.)
This infimim is attained  uniquely at the drift
\begin{equation}\label{bestb}
b_0(x)=\begin{cases}+\infty,\ x<A_-(\mu);\\
D\Big(\frac14\frac{\overline{\mu}_-\thinspace'(x)}{\overline{\mu}_-(x)}-\frac{|\log p|}{2\int_{-\infty}^0\overline{\mu}_-^\frac12(y)dy}\thinspace\overline{\mu}_-^\frac12(x)\Big),\ A_-(\mu)<x<0;\\
D\Big(\frac14\frac{\overline{\mu}_+\thinspace'(x)}{\overline{\mu}_+(x)}+\frac{|\log(1-p)|}{2\int_0^\infty\overline{\mu}_+^\frac12(y)dy}\thinspace\overline{\mu}_+^\frac12(x)\Big),\ 0<x<A_+(\mu);\\
-\infty,\ x>A_+(\mu).
\end{cases}
\end{equation}
If $\int_{A_-(\mu)}\overline{\mu}_-^{\thinspace-\frac12}(x)dx=\infty$
($\int_{A_+(\mu)}\overline{\mu}_+^{\thinspace-\frac12}(x)dx=\infty$),
then this drift prevents the diffusion $X(\cdot)$ from reaching $A_-(\mu)$ $\big(A_+(\mu)\big)$, and
thus there is no need to define $b_0$ on
 $(-\infty, A_-(\mu)) \big((A_+(\mu),\infty)\big)$.
Otherwise the diffusion $X(\cdot)$ can reach  $A_-(\mu)$ $\big(A_+(\mu)\big)$, and the drift of $+\infty$ to the left of $A_-(\mu)$ ($-\infty$ to the right of $A_+(\mu)$) causes the diffusion to be reflected there.

\noindent iii. For those $\mu$ for which the infimum in (i) is not attained, the infimum is approached by a sequence $\{b_n\}_{n=1}^\infty$ of drifts, with $b_n$
given by \eqref{bestb} with $\mu=(1-p)\mu_-+p\mu_+$ replaced by $\mu_n=(1-p)\mu_{-;n}+p\mu_{+;n}$, where $\mu_n$  satisfies the square root balance condition \eqref{srbc},
is of the type described in (ii) and
 converges weakly to $\mu$.
\end{theorem}
\bf \noindent Remark 1.\rm\ After this paper was competed and placed on the Mathematics ArXiv, I was directed to
 \cite{K} by one of its coauthors. That paper, which appears in the physics literature, treats the same problem
  considered here. In particular, in the case that $\mu$ is symmetric and possesses a density, the authors found that  $b_0$ from
 \eqref{bestb} (with $p=\frac12$ and $\mu_+(x)=\mu_-(-x)$) is a critical point of the map $b\to\int_{\mathbb{R}} (E^{(b)}_0T_a)\thinspace\mu(da)$, and they
 calculated the corresponding expected search time, obtaining   the expression on the righthand side of
\eqref{symmetriccaseinf}. They   stated that this search time is optimal.

\medskip

\noindent \bf Remark 2.\rm\ Let
\begin{equation*}
\text{EV}(\mu_-):=\int_{-\infty}^0x\mu_-(dx),\ \text{EV}(\mu_+):=\int_0^\infty x\mu_+(dx)
\end{equation*}
denote respectively the expected values of random variables distributed according to $\mu_-$ and according to $\mu_+$.
Since $|\text{EV}(\mu_-)|=\int_{-\infty}^0\overline{\mu}_-(x)dx$ and $\text{EV}(\mu_+)=
\int_0^\infty\overline{\mu}_+(x)dx$, it follows from part (i) of the theorem that
\begin{equation*}
\inf_{b\in \mathcal{D}_\mu}\int_{\mathbb{R}} (E^{(b)}_0T_a)\thinspace\mu(da)\ge
\frac2D\Big(\frac{1-p}{|\log p| }(\text{EV}(\mu_-))^2+\frac p{|\log(1-p)|}(\text{EV}(\mu_+))^2\Big),
\end{equation*}
with equality if and only if $\mu_-$ and $\mu_+$ are the degenerate probability measures $\delta_{A_-(\mu)}$ and $\delta_{A_+(\mu)}$ respectively.
In particular, in the case that the target distribution $\mu$ is symmetric,  then $\text{AvgDist}(\mu):=\text{EV}(\mu_+)$ is the expected distance of the target to the origin, and
\begin{equation}\label{infavgdist}
\inf_{b\in \mathcal{D}_\mu}\int_{\mathbb{R}} (E^{(b)}_0T_a)\thinspace\mu(da)\ge\frac2{D\log2}(\text{AvgDist}(\mu))^2,
\end{equation}
with equality if and only if the target distribution is $\mu=\frac12\delta_{-A}+\frac12\delta_A$, where $A=-A_-(\mu)=A_+(\mu)$.
In section \ref{examples}, it is shown that for a number of families of symmetric distributions,
the ratio of $\inf_{b\in \mathcal{D}_\mu}\int_{\mathbb{R}} (E^{(b)}_0T_a)\thinspace\mu(da)$ to
$(\text{AvgDist}(\mu))^2$ is constant within each family, that is, independent  of the particular parameter.
\medskip

\noindent
\bf Remark 3.\rm\ Note that in part (ii), if $\mu_+$ does not have an
atom at $A_+(\mu)$ and  has a density that is differentiable at $A_+(\mu)$, then
this density vanishes at least to order one, and thus $\overline{\mu}_+$ vanishes at least to order two. Thus,
$\int_{A_+(\mu)}\overline{\mu}_+^{\thinspace-\frac12}(x)dx=\infty$, and
the diffusion with optimal drift $b_0$ cannot reach $A_+(\mu)$. However, if $\mu_+$ has an atom at $A_+(\mu)$, or if it doesn't
have an atom at $A_+(\mu)$ and its density vanishes to an order less than one at $A_+(\mu)$, then
$\int_{A_+(\mu)}\overline{\mu}_+^{\thinspace-\frac12}(x)dx<\infty$, and the diffusion with optimal drift can reach $A_+(\mu)$.
The same considerations hold at $A_-(\mu)$.

\medskip

In section \ref{examples} we  illustrate Theorem \ref{thm2} with a number of examples.
\medskip

We now turn to the case that the target distribution does not satisfy the square root balance
condition \eqref{srbc}. Here we have only partial results.
\begin{theorem}\label{thm3}
Let the target distribution $\mu$ satisfy $\mu=(1-p)\mu_-\thinspace +p\mu_+$  as in $\eqref{mu}$, let $A_-(\mu)$ and $A_+(\mu)$ be as in \eqref{support} and let
$\overline{\mu}_-(x)$ and $\overline{\mu}_+(x)$ be as in \eqref{tails}. Let the class of admissible drifts $\mathcal{D}_\mu$ be as in \eqref{admiss}.
Assume also that the target distribution does not satisfy the square root balance condition \eqref{srbc}, but
that  $\int_0^\infty\overline\mu_+^\frac12(x)dx$ and $\int_{-\infty}^0\overline{\mu}_-^\frac12(x)dx$ are finite.
Then

\noindent i.  $\inf_{b\in \mathcal{D}_\mu}\int_{\mathbb{R}} (E^{(b)}_0T_a)\thinspace\mu(da)$ is not attained.

\noindent ii.
\begin{equation}\label{notinf}
\begin{aligned}
&\inf_{b\in \mathcal{D}_\mu}\int_{\mathbb{R}} (E^{(b)}_0T_a)\thinspace\mu(da)<\\
&\frac2D\Big(\frac{1-p}{|\log p| }\thinspace\big(\int_{-\infty}^0\overline{\mu}_-^\frac12(x)dx\big)^2+\frac p{|\log(1-p)|}\thinspace\big(\int_0^\infty\overline{\mu}_+^\frac12(x)dx\big)^2\Big)-\\
&\frac2D\Big(\frac{1-p}{|\log p| }\thinspace\int_{-\infty}^0\overline{\mu}_-^\frac12(x)dx-\frac p{|\log(1-p)|}\thinspace\int_0^\infty\overline{\mu}_+^\frac12(x)dx\Big)^2.
\end{aligned}
\end{equation}

\noindent iii. If  $\mu$ restricted to $(A_-(\mu),A_+(\mu))$   is absolutely continuous with a piecewise continuous, locally bounded density on $(A_-(\mu),A_+(\mu))$
($\mu$ may possess an atom at  $A_-(\mu)$ and/or at  $A_+(\mu)$), then $\int_{\mathbb{R}} (E^{(b)}_0T_a)\thinspace\mu(da)$ is equal to the righthand side of \eqref{notinf} when
$b$ is given by \eqref{bestb}.

\end{theorem}
\noindent\bf Remark 1.\rm\ Note that the expression on the third line of \eqref{notinf} would be zero if the square root balance condition held, in which case the right hand side of \eqref{notinf} would be equal
to the right hand side of \eqref{infb}.
 The right hand side of \eqref{notinf} can also be written as
$$
\begin{aligned}
&(\frac{1-p}{|\log p|})(1-\frac{1-p}{|\log p|})\big(\int_{-\infty}^0\overline{\mu}_-^\frac12(x)dx\big)^2+\\
&(\frac{p}{|\log(1- p)|})(1-\frac{p}{|\log(1- p)|})\big(\int_0^\infty\overline{\mu}_+^\frac12(x)dx\big)^2+\\
&2\frac{p(1-p)}{|\log(1-p)||\log p|}(\int_{-\infty}^0\overline{\mu}_-^\frac12(x)dx)(\int_0^\infty\overline{\mu}_+^\frac12(x)dx).
\end{aligned}
$$
It is easy to check that the coefficients of $\big(\int_{-\infty}^0\overline{\mu}_-^\frac12(x)dx\big)^2$
and $\big(\int_0^\infty\overline{\mu}_+^\frac12(x)dx\big)^2$ in the above expression are positive.
\medskip

The above results suggest two open problems.

\noindent \bf Open Problem 1.\rm\ In the case that the square root balance condition fails, calculate
$\inf_{b\in \mathcal{D}_\mu}\int_{\mathbb{R}} (E^{(b)}_0T_a)\thinspace\mu(da)$.

\noindent \bf Open Problem 2.\rm\ Is $\inf_{b\in \mathcal{D}_\mu}\int_{\mathbb{R}} (E^{(b)}_0T_a)\thinspace\mu(da)$ necessarily infinite in the case
that one out of $\int_{-\infty}^0\overline{\mu}_-^\frac12(x)dx$ and  $\int_0^\infty\overline{\mu}_+^\frac12(x)dx$ is infinite and the other is finite?
If not, what can be said
about  $\inf_{b\in \mathcal{D}_\mu}\int_{\mathbb{R}} (E^{(b)}_0T_a)\thinspace\mu(da)$?
\medskip

It is natural to wonder about the corresponding problem in higher dimensions. Let the unknown stationary target $a\in \mathbb{R}^d$ be distributed according to a known distribution
$\mu\in \mathcal{P}(\mathbb{R}^d)$, the space of probability measures on $\mathbb{R}^d$.
Consider a diffusion process $X(\cdot)$ starting at 0 and  generated by $\frac D2\Delta +b(x)\cdot\nabla$, and denote probabilities and expectations with respect to this process by $P_0^{(b)}$ and $E_0^{(b)}$.
Let $\epsilon>0$ and define $\tau_{a;\epsilon}=\inf\{t\ge0:|X(t)-a|\le \epsilon\}$
One then wants to minimize $\int_{\mathbb{R}}(E_0^{(b)}\tau_{a;\epsilon})\mu(da)$ over a natural class of admissible drifts.  In the two-dimensional case,  resolve the drift into radial and angular components, $r$ and $\theta$,
and write $b(x)\cdot\nabla=b_{\thinspace\text{rad}}(r,\theta)\frac \partial {\partial r}+b_{\thinspace\text{ang}}(r,\theta)\frac1r\frac \partial{\partial\theta}$. It is intuitively clear that if we let $b_{\thinspace\text{rad}}(r,\theta)$ depend only on $r$ and let $b_{\thinspace\text{ang}}(r,\theta)$ be equal to a constant $b_{\thinspace\text{ang}}$, then for $|a|-\epsilon>0$,
the quantity $\lim_{\thinspace b_{\thinspace\text{ang}}\to\infty}E^{(b)}_0\tau_\epsilon$ will just be equal to the expected hitting time of $a-\epsilon$
 for the one-dimensional radial diffusion started from $0^+$ and generated by
 $\frac{d^2}{dr^2}+\frac1r\frac d{dr}+b_{\thinspace\text{rad}}(r)\frac d{dr}$.
And this latter hitting time converges to 0 as the drift $b_{\text{rad}}(r)$ converges pointwise to $\infty$.
Thus, in order to obtain something interesting, a   restriction must be placed on the angular  drift. Such a limitation doesn't seem to  occur in higher dimensions.
In any case, perhaps a good starting point would be to consider the class of radial drifts.
Our intuition is that the higher the dimension, the more strongly toward the origin
 will point an optimal or near-optimal radial drift, since the higher the dimension, the more space there is to search at each fixed radius.
 Of course, the great difficulty  with the multi-dimensional case is that there isn't an explicit formula for $E_0^{(b)}\tau_{a;\epsilon}$.
\medskip

We conclude this introductory section with a sketch
of our method of approach to the variational problem, $\inf_{b\in \mathcal{D}_\mu}\int_{\mathbb{R}} (E^{(b)}_0T_a)\thinspace\mu(da)$, since it  has a certain novelty to it.
To proceed, we need the following proposition.
\begin{proposition}\label{exphit}
Let $\mu\in\mathcal{P}(\mathbb{R})$,  and let $b\in\mathcal{D}_\mu$, where $\mathcal{D}_\mu$ is  the class of admissible drifts as in \eqref{admiss}. Then
\begin{equation}\label{exphitting}
E^{(b)}_0T_a=\begin{cases}
\frac2D\int_a^0dx\exp(-\int_0^x\frac2Db(y)dy)\int_x^{A_+(\mu)}dz\exp(\int_0^z\frac2Db(t)dt),\ A_-(\mu)\le a<0;\\
 \frac2D\int_0^adx\exp(-\int_0^x\frac2Db(y)dy)\int_{A_-(\mu)}^xdz\exp(\int_0^z\frac2Db(t)dt),\ 0<a\le A_+(\mu).
\end{cases}
\end{equation}
\end{proposition}
\bf\noindent Remark.\rm\ The explicit formula for the hitting time in Proposition \ref{exphit} is of course not new, but since we need it for a variety of situations---including the case in which the drift can blow up at the boundary, and including
the case of reflection at the boundary, we will present its proof in section \ref{pfexphit}.

In light of Proposition \ref{exphit}, for $\mu=(1-p)\mu_-+p\mu_+$, we have
\begin{equation}\label{beforeRS}
\begin{aligned}
&\frac D2\int_{\mathbb{R}} (E^{(b)}_0T_a)\thinspace\mu(da)=\\
&(1-p)\int_{A_-(\mu)}^0\mu_-(da)\Big[\int_a^0dx\exp(-\int_0^x\frac2Db(y)dy)\int_x^{A_+(\mu)}dz\exp(\int_0^z\frac2Db(t)dt)\Big]+\\
&p\int_0^{A_+(\mu)}\mu_+(da)\Big[\int_0^adx\exp(-\int_0^x\frac2Db(y)dy)\int_{A_-(\mu)}^xdz\exp(\int_0^z\frac2Db(t)dt)\Big],
\end{aligned}
\end{equation}
and after a Reimann-Stieltjes integration by parts, we obtain
\begin{equation}\label{bform}
\begin{aligned}
&\frac D2\int_{\mathbb{R}} (E^{(b)}_0T_a)\thinspace\mu(da)=\\
&(1-p)\int_{A_-(\mu)}^0da\thinspace\overline{\mu}_-(a)\Big[\exp(-\int_0^a\frac2Db(y)dy)\int_a^{A_+(\mu)}dz\exp(\int_0^z\frac2Db(t)dt)\Big]+\\
&p\int_0^{A_+(\mu)}da\thinspace\overline{\mu}_+(a)\Big[\exp(-\int_0^a\frac2Db(y)dy)\int_{A_-(\mu)}^adz
\exp(\int_0^z\frac2Db(t)dt)\Big].
\end{aligned}
\end{equation}
In the case that $A_+(\mu)<\infty$ ($A_-(\mu)>-\infty$),  the passage from \eqref{beforeRS} to \eqref{bform} is true
even if  $\mu_+$  ($\mu_-$) has an atom at   $A_+(\mu)$ ($A_-(\mu)$), or if
$\int_0^{A_+(\mu)}dx\exp(-\int_0^x\frac2Db(y)dy)=\infty$  ($\int_{A_-(\mu)}^0dx\exp(-\int_0^x\frac2Db(y)dy)=\infty$).
This is because in \eqref{tails},  $\overline{\mu}_+(x)$ ($\overline{\mu}_-(x)$) has been defined not to include
$\mu_+(\{x\})$  ($\mu_-(\{x\}))$.
In the case that $A_+(\mu)=\infty$ ($A_-(\mu)=-\infty$), the passage from \eqref{beforeRS} to \eqref{bform} is true for the following reason. (We explain it for $A_+(\mu)=\infty$.)
We need to justify having  ignored in the integration by parts the possible contribution
\begin{equation}\label{possiblecontr}
\lim_{A\to\infty}\overline{\mu}_+(A)\int_0^Adx\exp(-\int_0^x\frac2Db(y)dy)\int_{A_-(\mu)}^xdz\exp(\int_0^z\frac2Db(t)dt).
\end{equation}
If the term
$\int_0^\infty da\thinspace\overline{\mu}_+(a)\Big[\exp(-\int_0^a\frac2Db(y)dy)\int_{A_-(\mu)}^adz
\exp(\int_0^z\frac2Db(t)dt)\Big]$  on the right hand side of \eqref{bform} is infinite, then nothing need be checked; thus,
 assume this  integral is finite.
Then we need to show that  \eqref{possiblecontr} is equal to 0. Since $\lim_{A\to\infty}\overline{\mu}_+(A)=0$,
 \eqref{possiblecontr} is equal to
 \begin{equation*}\label{possiblecontragain}
\lim_{A\to\infty}\overline{\mu}_+(A)\int_{A_0}^Adx\exp(-\int_0^x\frac2Db(y)dy)\int_{A_-(\mu)}^xdz\exp(\int_0^z\frac2Db(t)dt),
\end{equation*}
 for
any fixed $A_0>0$.
We have
\begin{equation*}\label{deltaA0A}
\begin{aligned}
&\overline{\mu}_+(A)\int_{A_0}^Adx\exp(-\int_0^x\frac2Db(y)dy)\int_{A_-(\mu)}^xdz\exp(\int_0^z\frac2Db(t)dt)
\le\\
&\int_{A_0}^Ada\thinspace\overline{\mu}_+(a)\Big[\exp(-\int_0^a\frac2Db(y)dy)\int_{A_-(\mu)}^adz
\exp(\int_0^z\frac2Db(t)dt)\Big]:=\delta(A_0,A).
\end{aligned}
\end{equation*}
By the integrability assumption, $\lim_{A_0\to\infty}\lim_{A\to\infty}\delta(A_0,A)=0$.
We conclude from the above argument that \eqref{possiblecontr} is indeed equal to 0.

We want to minimize the righthand side of \eqref{bform} over $b\in\mathcal{D}_\mu$.
There are two points of view that one can take, and it  turns out that  both of them are essential.
One point of view is to consider the righthand side of \eqref{bform} as a functional of $b$; we will call it $G_1$:
\begin{equation}\label{G1}
\begin{aligned}
&G_1(b)=(1-p)\int_{A_-(\mu)}^0da\thinspace\overline{\mu}_-(a)\Big[\exp(-\int_0^a\frac2Db(y)dy)\int_a^{A_+(\mu)}dz\exp(\int_0^z\frac2Db(t)dt)\Big]+\\
&p\int_0^{A_+(\mu)}da\thinspace\overline{\mu}_+(a)\Big[\exp(-\int_0^a\frac2Db(y)dy)\int_{A_-(\mu)}^adz\exp(\int_0^z\frac2Db(t)dt)\Big].
\end{aligned}
\end{equation}
For the other point of view, define the distribution function
\begin{equation}\label{def-F}
F(x)=\frac{\int_{A_-(\mu)}^xdz\exp(\int_0^z\frac2Db(t)dt)}{\int_{A_-(\mu)}^{A_+(\mu)}dz\exp(\int_0^z\frac2Db(t)dt)},
\end{equation}
and let $f(x)=F'(x)$ denote its density.
Then the righthand side of \eqref{bform} can be thought of as a functional of $F$; we call it $G_2(F)$. It is given by
\begin{equation}\label{G2}
G_2(F)=(1-p)\int_{A_-(\mu)}^0\overline{\mu}_-(a)\frac{F(A_+(\mu))-F(a)}{f(a)}da
+p\int_0^{A_+(\mu)}\overline{\mu}_+(a)\frac{F(a)}{f(a)}da.
\end{equation}
Of course, $F(A_+(\mu))=1$, but it is useful to write it as we have done in order to exploit the homogeneity.
Indeed, note that now we can consider $G_2$ to be a functional of positive multiples of distribution functions of the type just described, and we have
$G_2(cF)=G_2(F)$, for all $c>0$.
We denote the domain of the functional $G_2$ by $\mathcal{D}(G_2)$ and specify it as follows:
\begin{equation}\label{domainG2}
\begin{aligned}
&\mathcal{D}(G_2)\ \text{is the set of positive multiples of  the class of distributions functions}\\
& F\ \text{that can be written in the form
\eqref{def-F}, where}\ b\in\mathcal{D}_\mu.
\end{aligned}
\end{equation}

To search for critical points, the first point of view requires us to consider the condition
\begin{equation}\label{badvariation}
0=\lim_{\epsilon\to0}\frac{G_1(b+\epsilon\beta)-G_1(b)}\epsilon,
\end{equation}
for an appropriate wide class of drifts $\beta$.
To isolate $\beta$ in \eqref{badvariation} requires numerous integration by parts. This eventually leads  to an equation of the form $(1-p)\int_{A_-(\mu)}^0\beta(a)\Psi_-(a)da+
 p\int_0^{A_+(\mu)}\beta(a)\Psi_+(a)da=0$, for all $\beta$, where $\Psi_-$ and $\Psi_+$ are expressions involving  $b$.
 Thus   $\Psi_-(a)=0$, for $A_-(\mu)<a<0$, and $\Psi_+(a)\equiv0$, for $0< a<A_+(\mu)$.
 However, we did not find it tractable to solve these equations for $b$.

Since $G_2$ is homogeneous of order zero,
to search for critical points via the second point of view we  consider the condition
\begin{equation}\label{goodvariation}
0=\lim_{\epsilon\to0}\frac{G_2(F+\epsilon Q)-G_2(F)}\epsilon
\end{equation}
(here $\epsilon$ takes on both positive and negative values), where $Q$ is such that $F+\epsilon Q$ belongs to the domain $\mathcal{D}(G_2)$ of $G_2$. In fact, in order to ensure that we can interchange the order of the integration and the differentiation when we   calculate
\eqref{goodvariation} with $G_2$ given by \eqref{G2},
 and also in order to ensure
that $F+\epsilon Q$ is positive for small negative $\epsilon$,
we will actually restrict ourselves to distribution functions  $Q$ with densities compactly   supported   in $\big(A_-(\mu),A_+(\mu)\big)$.
After integrating by parts several times to isolate the density  $q:=Q'$  of $Q$, we obtain an equation of the form
$$
(1-p)\int_{A_-(\mu)}^0q(a)\Phi_-(a)da+ p\int_0^{A_+(\mu)}q(a)\Phi_+(a)da=\Sigma(F,\overline{\mu}_-,\overline{\mu}_+),
$$
where $\Phi_-$ is an expression involving $F,F'$ and  $\overline{\mu}_-$,  $\Phi_+$ is an expression
involving $F,F'$ and $\overline{\mu}_+$, and
$\Sigma$ is a constant involving $F'$,  $\overline{\mu}_-$ and $\overline{\mu}_+$. Since $q$ is a general compactly supported  density function, this leads to the equations
$(1-p)\Phi_-(a)=\Sigma(F,\overline{\mu}_-,\overline{\mu}_+)$, for $A_-(\mu)<a<0$ and $p\Phi_+(a)=\Sigma(F,\overline{\mu}_-,\overline{\mu}_+)$, for $0<a<A_+(\mu)$.
These equations for $F$ turn out to be tractable.
If $\mu$ satisfies the square root  balance condition and is as in (ii) of Theorem \ref{thm2},
then there is a unique solution $F_0$ for which the corresponding $b_0$
(obtained via $\frac{F_0''(x)}{F_0'(x)}=\frac{f_0'}{f_0}(x)=\frac2D b_0(x)$) is in $\mathcal{D}_\mu$; otherwise there
is no solution, and thus there are no critical points.

When $G_2$ possesses a critical point $F_0$, how do we show that in fact $G_2$ attains its   global minimum uniquely    at $F_0$?
(Or equivalently, how do we show  that the global minimum of $G_1$ is attained uniquely at $b_0$, where $b_0$ corresponds to $F_0$ via \eqref{def-F}?)
Uniqueness is immediate since there is only one critical point.
Due to certain technical obstacles, we can only  show  directly
that $F_0$ is the global minimum
 in the case of  measures
$\mu$ for which $A_-(\mu)$ and $A_+(\mu)$ are finite and are
atoms of the measure.
The  case of a general measure  is obtained by approximating by measures as above.
 It is natural to take an arbitrary admissible $F$ and
consider  $L_2(t):=G_2((1-t)F_0+tF)$. We would like to show that  $L_2$ is  convex and that $L_2'(0)=0$, from which it would follow that the global minimum is attained at $F_0$.
 However,
we see no way to prove that $L_2$ is convex.
On the other hand, if we consider $L_1(t):=G_1((1-t)b+t\beta)$, for arbitrary $\beta$ and  \it arbitrary\rm\ $b$, not just for the corresponding critical case $b=b_0$, then it is very simple to show that $L_1$ is convex.
\begin{proposition}\label{convex}
The set $\mathcal{D}_\mu$ is convex.
Let $b,\beta\in \mathcal{D}_\mu$, and let
 $L_1(t)=G_1((1-t)b+t\beta)$,\ $0\le t\le 1$, where $G_1$ is as in \eqref{G1}. Then $L_1$ is convex.
\end{proposition}
Our proof of the above result does not require that the measure be of the special type mentioned above.
However, we require this  restriction to  prove  the following technical result.
\begin{proposition}\label{G1G2}
Assume that $A_-(\mu)$ and $A_+(\mu)$ are finite,  that   $\mu$ has atoms at both $A_-(\mu)$ and $A_+(\mu)$, and that its
  restriction to $(A_-(\mu),A_+(\mu))$   is absolutely continuous with a
  piecewise continuous, locally bounded density. Let
  $b_0$ be as in \eqref{bestb}.
Let $b\in\mathcal{D}_\mu$, and define $L_1(t)=G_1((1-t)b_0+tb),\ 0\le t\le 1$, where $G_1$ is as in \eqref{G1}.
Then $L_1'(0)=0$.
\end{proposition}
From the above two propositions, it follows easily that when $\mu$ is as in Proposition \ref{G1G2},  the critical point $F_0$ is in fact the global minimum.

The rest of the paper is organized as follows. In section \ref{examples}, we illustrate Theorem \ref{thm2} with a number of examples.
The proof of Theorem \ref{thm1} requires the result in Theorem
\ref{thm2}-i, and  the proof of Theorem \ref{thm3} requires some of the proof of Theorem \ref{thm2}.
Thus we first prove Theorem \ref{thm2} in section \ref{pfthm2}, and then prove Theorems \ref{thm1} and  \ref{thm3} in
section \ref{pfthm13}.
Of course, these result also depend on Propositions \ref{exphit}, \ref{convex} and \ref{G1G2}.
The first of these propositions is proved in section \ref{pfexphit} and the next two are proved in section \ref{pfprop34}.

\section{Some examples of Theorem \ref{thm2}}\label{examples}
We give several examples to illustrate Theorem \ref{thm2}, restricting always to the case that the target distribution $\mu$ is symmetric.
Recall that in the symmetric case, the infimum is given by \eqref{symmetriccaseinf}.
Recall also from Remark 2 after Theorem \ref{thm2} that in the symmetric case, the expected distance of the target is equal to $\int_0^\infty\overline{\mu}_+(dx)$, and has been denoted
by AvgDist$(\mu)$.
Furthermore, by \eqref{infavgdist}, the ratio
$\frac{D\log2}2\thinspace\frac{\inf_{b\in \mathcal{D}_\mu}\int_{\mathbb{R}} (E^{(b)}_0T_a)\thinspace\mu(da)}
{(\text{AvgDist}(\mu))^2}$  is always greater or equal to 1, with equality only in the case of the distributions
in example I below.
\medskip

\noindent \bf I.\it\ Symmetric Degenerate Distribution\rm: $\mu=\frac12\delta_{-A}+\frac12\delta_A,\ A>0$

We have $\overline{\mu}_+(x)=1,\ x\in[0,A)$, and   $\overline{\mu}_+(x)=0, \ x\ge A.$
Thus,
$$
\inf_{b\in \mathcal{D}_\mu}\int_{\mathbb{R}} (E^{(b)}_0T_a)\thinspace\mu(da)=\frac2{D\log 2}A^2=\frac2{D\log 2}(\text{AvgDist}(\mu))^2.
$$
The infimum is attained at the anti-symmetric drift $b_0$ satisfying
$$
b_0(x)=\begin{cases}\frac{D\log 2}{2A},\ 0<x<A;\\ -\infty,\ x>A.\end{cases}
$$
Of course, the corresponding diffusion can reach $\pm A$.

\medskip

\noindent \bf II.\it\ Symmetric Uniform Distribution\rm: $\mu=\text{U}([-A,A]), \  A>0$

We have $\overline{\mu}_+(x)=1-\frac xA$, \ $x\in[0,A)$, and $\overline{\mu}_+(x)=0,\ x>A$.
One has $\int_0^\infty\overline{\mu}_+^\frac12(x)dx=\frac1{A^\frac12}\int_0^A(A-x)^\frac12dx=\frac23 A.$
Also, $\text{AvgDist}(\mu)=\frac A2$.
Thus,
$$
\inf_{b\in \mathcal{D}_\mu}\int_{\mathbb{R}} (E^{(b)}_0T_a)\thinspace\mu(da)=\frac{8}{9D\log 2}A^2=\frac{16}9\frac2{D\log 2}(\text{AvgDist}(\mu))^2.
$$
\medskip
The infimum is attained at the anti-symmetric drift $b_0$ satisfying
$$
b_0(x)=\begin{cases}D\big[-\frac1{4(A-x)}+\frac{3\log 2}{4A}(1-\frac xA)^\frac12\big],\ x\in(0,A);\\ -\infty,\ x>A.\end{cases}
$$
Despite the unbounded drift, the corresponding diffusion can reach $\pm A$.
\medskip

\noindent\bf III.\it\ Symmetric Exponential Distribution\rm: $\mu=\frac12\text{Exp}(\lambda)+\frac12\big(-\text{Exp}(\lambda)\big),\ \lambda>0$

We have  $\overline{\mu}_+(x)=e^{-\lambda}$,\ $x>0$, and $\text{AvgDist}(\mu)=\frac1\lambda$. Thus,
$$
\inf_{b\in \mathcal{D}_\mu}\int_{\mathbb{R}} (E^{(b)}_0T_a)\thinspace\mu(da)=\frac8{D\lambda^2\log 2}=4\frac2{D\log 2}(\text{AvgDist}(\mu))^2.
$$
The infimum is attained at the anti-symmetric drift $b_0$ satisfying
$$
b_0(x)=D\big(-\frac\lambda4+\frac\lambda4(\log2) e^{-\frac\lambda2 x}\big),\ x>0.
$$
\medskip

\bf\noindent IV.\it\ Symmetric Gaussian Distribution\rm: $\mu=N(0,\sigma^2)$

We have $\overline{\mu}_+(x)=\int_x^\infty\frac{\exp(-\frac{y^2}{2\sigma^2})}{\sqrt{2\pi}\sigma}dy=1-\Phi(\frac x\sigma$),
where $\Phi(z)=\int_{-\infty}^z\frac{\exp(-\frac{y^2}{2})}{\sqrt{2\pi}}dy$.
One has  $\int_0^\infty\overline{\mu}_+^\frac12(x)dx=\sigma\int_0^\infty\big(1-\Phi(z)\big)^\frac12dz\approx0.9219\thinspace\sigma$.
Also, $\text{AvgDist}(\mu)=\frac\sigma{\sqrt{2\pi}}$.
Thus,
$$
\begin{aligned}
&\inf_{b\in \mathcal{D}_\mu}\int_{\mathbb{R}} (E^{(b)}_0T_a)\thinspace\mu(da)=\frac{2\sigma^2\big(\int_0^\infty\big(1-\Phi(z)\big)^\frac12dz\big)^2}{D\log 2}=\\
&2\pi\big(\int_0^\infty\big(1-\Phi(z)\big)^\frac12dz\big) ^2\frac{2}{D\log 2}(\text{AvgDist}(\mu))^2\approx
5.340\frac{2}{D\log 2}(\text{AvgDist}(\mu))^2.
\end{aligned}
$$
The infimum is attained at the anti-symmetric drift $b_0$ satisfying
$$
b_0(x)=D\big[-\frac14\frac{e^{-\frac{x^2}{2\sigma^2}}}{\sqrt{2\pi}\sigma(1-\Phi(\frac x\sigma))}+\frac{\log2}{2\sigma\int_0^\infty\big(1-\Phi(z)\big)^\frac12dz}\big(1-\Phi(\frac x\sigma)\big)^\frac12\big],\ x>0.
$$
\medskip

\noindent \bf V.\it\ Symmetric Pareto Distribution\rm: $\mu=\frac12 \text{Pareto}(\alpha,A_0)+\frac12(-\text{Pareto}(\alpha,A_0))$, where $A_0>0$, $\alpha>2$ and
$\mu_+\sim\text{Pareto}(\alpha,A_0)$ is given by $\overline\mu_+(x)=\min(1,(\frac x{A_0})^{-\alpha})$, \ $x>0$.

One has  $\int_0^\infty\overline{\mu}_+^\frac12(x)dx=A_0+\frac2{\alpha-2}$
 and $\text{AvgDist}(\mu)=A_0+\frac1{\alpha-1}$. Thus
$$
\inf_{b\in \mathcal{D}_\mu}\int_{\mathbb{R}} (E^{(b)}_0T_a)\thinspace\mu(da)=
\big(\frac{\alpha-1}{\alpha-2}\big)^2\Big(\frac{A_0(\alpha-2)+2}{A_0(\alpha-1)+1}\Big)^2\frac2{D\log 2}(\text{AvgDist}(\mu))^2.
$$
The infimum is attained at the anti-symmetric drift $b_0$ satisfying
$$
b_0(x)=\begin{cases} D\frac{\log2}{2(A_0+\frac2{\alpha-2})}, \ x\in(0,A_0);\\
D\big(-\frac \alpha{4x}+\frac{\log 2}{2(A_0+\frac2{\alpha-2})}(\frac x{A_0})^{-\frac \alpha2}\big),\ x>A_0.
 \end{cases}
$$
Note that this drift is only piecewise continuous, because $\overline{\mu}_+$ is only piecewise continuously differentiable.

\medskip

\bf \noindent Remark 1.\rm\ Note that the ratio
$\frac{D\log2}2\thinspace\frac{\inf_{b\in \mathcal{D}_\mu}\int_{\mathbb{R}} (E^{(b)}_0T_a)\thinspace\mu(da)}
{(\text{AvgDist}(\mu))^2}$ is independent of the parameter for each of the families of distributions in examples I-IV above.
For the family of Pareto distributions in example V, for fixed $A_0$, this ratio
 increases from $1^+$ to $\infty$
as $\alpha$ decreases from  $\infty$ to $2^+$.

\bf\noindent Remark 2.\rm\ Note the asymptotic behavior as $x\to\infty$ of the minimizing drift $b_0(x)$  in examples III--V:

\noindent Exponential: $\lim_{x\to\infty}b(x)=-\frac{\lambda}4D$;

\noindent Gaussian: $b(x)\sim-\frac{x}{4\sigma^2}D$;

\noindent Pareto: $b(x)\sim-\frac{\alpha}{4x}D$.

\section{Proof of Theorem \ref{thm2}}\label{pfthm2}
 We begin with the long proof of part (ii). The proofs of the other two  parts use the result of part (ii).

\noindent \it Proof of part (ii).\rm\ Recalling \eqref{bform}-\eqref{G2} and recalling the
definition of  $\mathcal{D}(G_2)$ from \eqref{domainG2},
we search for critical points $F\in\mathcal{D}(G_2)$
of the functional $G_2(F)$.
Let   $Q$ denote an arbitrary distribution function on $(A_-(\mu),A_+(\mu))$, with a density $q$ that is
continuous, piecewise continuously differentiable and compactly supported in $(A_-(\mu),A_+(\mu))$.
Then $F+\epsilon Q$ belongs to the domain $\mathcal{D}(G_2)$ for all $\epsilon$ with sufficiently small absolute value.
To prove this, one needs to find  a $b_\epsilon\in\mathcal{D}_\mu$, the class of admissible drifts, such that
$$
\frac{F(a)+\epsilon Q(a)}{1+\epsilon}=
\frac{\int_{A_-(\mu)}^xdz\exp(\int_0^z\frac2Db_\epsilon(t)dt)}
{\int_{A_-(\mu)}^{A_+(\mu)}dz\exp(\int_0^z\frac2Db_\epsilon(t)dt)}.
$$
This can be solved directly for $b_\epsilon$ by differentiating, taking logarithms and then differentiating again.
(The conditions above on $q$ are dictated by the conditions on $b_\epsilon\in \mathcal{D}_\mu$.)

We call $F$ a critical point if
\eqref{goodvariation} holds for all such $Q$.
A necessary condition for
$\inf_{b\in \mathcal{D}_\mu}\int_{\mathbb{R}} (E^{(b)}_0T_a)\thinspace\mu(da)$ to be attained at some particular $b$ is that the corresponding $F$ (via \eqref{def-F}) is critical for $G_2$.
Indeed, if $F$ is not critical, then for some $\epsilon$ with small absolute value, we will have
$G_2(F+\epsilon Q)<G_2(F)$, or equivalently,
$\int_{\mathbb{R}}(E^{(b)}_0T_a)\thinspace\mu(da)=G_1(b)>G_1(b_\epsilon)=\int_{\mathbb{R}}(E^{(b_\epsilon)}_0T_a)\thinspace\mu(da)$.

Now $F$ will be critical, that is,  \eqref{goodvariation} will hold for all such $Q$,  if and only if
\begin{equation}\label{var-1}
\begin{aligned}
&(1-p)\int_{A_-(\mu)}^0\overline{\mu}_-(a)\Big(\frac{1-Q(a)}{f(a)}-\frac{(1-F(a))q(a)}{f^2(a)}\Big)da+\\
&p\int_0^{A_+(\mu)}\overline{\mu}_+(a)\Big(\frac{Q(a)}{f(a)}-\frac{F(a)q(a)}{f^2(a)}\Big)da=0,
\end{aligned}
\end{equation}
for all such $Q$.
Integration by parts gives
\begin{equation}\label{IBP-1}
\int_{A_-(\mu)}^0\overline{\mu}_-(a)\frac{Q(a)}{f(a)}da=\int_{A_-(\mu)}^0q(a)\big(\int_a^0\frac{\overline{\mu}_-(x)}{f(x)}dx\big)da
\end{equation}
and
\begin{equation}\label{IBP-2}
\int_0^{A_+(\mu)}\overline{\mu}_+(a)\frac{Q(a)}{f(a)}da=-\int_0^{A_+(\mu)}q(a)\big(\int_0^a\frac{\overline{\mu}_+(x)}{f(x)}dx\big)da
+\int_0^{A_+(\mu)}\frac{\overline{\mu}_+(a)}{f(a)}da.
\end{equation}
Substituting \eqref{IBP-1} and \eqref{IBP-2} into \eqref{var-1} gives
\begin{equation}\label{var-2}
\begin{aligned}
&(1-p)\int_{A_-(\mu)}^0q(a)\Big[-\int_a^0\frac{\overline{\mu}_-(x)}{f(x)}dx-\frac{\overline{\mu}_-(a)(1-F(a))}{f^2(a)}\Big]da+\\
&p\int_0^{A_+(\mu)}q(a)\Big[-\int_0^a\frac{\overline{\mu}_+(x)}{f(x)}dx-\frac{\overline{\mu}_+(a)F(a)}{f^2(a)}\Big]+\\
&(1-p)\int_{A_-(\mu)}^0\frac{\overline{\mu}_-(a)}{f(a)}da+p\int_0^{A_+(\mu)}\frac{\overline{\mu}_+(a)}{f(a)}da=0.
\end{aligned}
\end{equation}
Now \eqref{var-2} will hold for all densities $q$ of the type described above if and only if
\begin{equation}\label{key1}
\begin{aligned}
&\int_a^0\frac{\overline{\mu}_-(x)}{f(x)}dx+\frac{\overline{\mu}_-(a)(1-F(a))}{f^2(a)}=
\int_{A_-(\mu)}^0\frac{\overline{\mu}_-(x)}{f(x)}dx+\frac p{1-p}\int_0^{A_+(\mu)}\frac{\overline{\mu}_+(x)}{f(x)}dx,\\ &a\in(A_-(\mu),0);\\
&\int_0^a\frac{\overline{\mu}_+(x)}{f(x)}dx+\frac{\overline{\mu}_+(a)F(a)}{f^2(a)}=
\frac{1-p}p\int_{A_-(\mu)}^0\frac{\overline{\mu}_-(x)}{f(x)}dx+\int_0^{A_+(\mu)}\frac{\overline{\mu}_+(x)}{f(x)}dx,\\ &a\in(0,A_+(\mu)).
\end{aligned}
\end{equation}
Denoting by $C$ the constant on the right hand side of the first equation above, we have
\begin{equation}\label{regularity}
\overline{\mu}_-(a)=\frac{f^2(a)}{1-F(a)}\big(C-\int_a^0\frac{\overline{\mu}_-(x)}{f(x)}dx\big),\ a\in(A_-(\mu),0).
\end{equation}

From \eqref{def-F} and the fact that $b\in \mathcal{D}_\mu$, it follows that $F$ is  continuously differentiable on
$(A_-(\mu),0)$, and that $f$ is continuous and piecewise continuously differentiable with a locally bounded
derivative on $(A_-(\mu),0)$.
Thus, we deduce from \eqref{regularity} that $\overline{\mu}_-$ is continuous and piecewise continuously
 differentiable with locally bounded derivative on $(A_-(\mu),0)$.
The same analysis shows that $\overline{\mu}_+$ is continuous and piecewise continuously differentiable with locally
bounded derivative on $(0,A_+(\mu))$.
We have thus shown that a necessary condition for the existence of a critical point is that
the restriction of $\mu$ to $(A_-(\mu),A_+(\mu))$ is absolutely continuous  with a  density that
is piecewise continuous  and locally bounded.
 In addition, $\mu$ might possibly possess an atom at $A_-(\mu)$ and/or at
$A_+(\mu)$.

We now continue our analysis under the assumption that $\mu$ satisfies the above noted necessary condition for a critical
point.
Then $\overline{\mu}_-$ and $\overline{\mu}_+$ are continuous and piecewise continuously differentiable with a locally
bounded derivative. In the analysis that follows, we implicitly assume that
$\overline{\mu}_-$ and $\overline{\mu}_+$ are continuously differentiable. However everything still
goes through under the weaker assumption that they are continuous and piecewise continuously differentiable
with locally bounded derivative. (See example V in section \ref{examples} for a case where
$\overline{\mu}_-$ and $\overline{\mu}_+$ are only piecewise continuously differentiable.)
Differentiating  \eqref{key1} gives
\begin{equation*}
\begin{aligned}
&2\thinspace\frac{\overline{\mu}_-(a)}{f(a)}+2\thinspace\frac{\overline{\mu}_-(a)(1-F(a))f'(a)}{f^3(a)}-
\frac{\overline{\mu}_-'(a)(1-F(a))}{f^2(a)}=0,\ a\in(A_-(\mu),0);\\
&-2\thinspace\frac{\overline{\mu}_+(a)}{f(a)}+2\thinspace\frac{\overline{\mu}_+(a)F(a)f'(a)}{f^3(a)}-
\frac{\overline{\mu}_+'(a)F(a)}{f^2(a)}=0,\ a\in(0,A_+(\mu).
\end{aligned}
\end{equation*}
Multiplying through  by $f$, and noting that $f=F'$ and $f'=F''$, we can rewrite the above equations as
\begin{equation}\label{key1-diff}
\begin{aligned}
&2\thinspace\overline{\mu}_-+2\thinspace\frac{\overline{\mu}_-(1-F)F''}{(F')^2}-
\frac{\overline{\mu}_-'(1-F)}{F'}=0,\ \ \text{on}\ (A_-(\mu),0);\\
&-2\thinspace\overline{\mu}_++2\thinspace\frac{\overline{\mu}_+FF''}{(F')^2}-
\frac{\overline{\mu}_+'F}{F'}=0\ \ \text{on}\ (0,A_+(\mu).
\end{aligned}
\end{equation}
Since
$$
\big(\frac F{F'}\big)'=1-\frac{FF''}{(F')^2}\ \ \text{and}\ \  \big(\frac{1-F}{F}\big)'=-1-\frac{(1-F)F''}{(F')^2},
$$
it follows that \eqref{key1-diff} is equivalent to
\begin{equation}\label{keydiffequ}
\begin{aligned}
&2\thinspace\big(\frac{1-F}{(1-F)'}\big)'+\frac{\overline{\mu}_-'}{\overline{\mu}_-}\thinspace\frac{1-F}{(1-F)'}=0\ \ \text{on} \ (A_-(\mu),0);\\
&2\thinspace\big(\frac{F}{F'}\big)'+\frac{\overline{\mu}_+'}{\overline{\mu}_+}\thinspace\frac F{F'}=0 \ \text{on}\ (0,A_+(\mu)).
\end{aligned}
\end{equation}

We now work with the second equation in \eqref{keydiffequ}. Substituting  $H=\frac F{F'}$, we obtain the linear equation
\begin{equation*}
2H'+\frac{\overline{\mu}_+'}{\overline{\mu}_+}\thinspace H=0.
\end{equation*}
Solving for $H$ gives  $H=\text{const.}\thinspace \overline{\mu}_-^{\thinspace-\frac12}$.
 Thus, $\frac{F'}F=\text{const.}\thinspace\overline{\mu}_-^\frac12$, and solving for $F$ gives
\begin{equation}\label{F+}
F(a)=\exp(-k_1\int_a^{A_+(\mu)}\overline{\mu}_+^{\thinspace\frac12}(x)dx),\ a\in(0,A_+(\mu)),
\end{equation}
for some $k_1>0$, and thus
\begin{equation}\label{f+}
f(a)=k_1\overline{\mu}_+^{\thinspace\frac12}(a)\exp(-k_1\int_a^{A_+(\mu)}\overline{\mu}_+^{\thinspace\frac12}(x)dx),\ a\in(0,A_+(\mu)).
\end{equation}
Analyzing the first equation in \eqref{keydiffequ} similarly, we arrive at
\begin{equation}\label{F-}
F(a)=1-\exp(-k_2\int_{A_-(\mu)}^a\overline{\mu}_-^{\thinspace\frac12}(x)dx),\ a\in(A_-(\mu),0).
\end{equation}
for some $k_2>0$, and thus
\begin{equation}\label{f-}
f(a)=k_2\overline{\mu}_-^{\thinspace\frac12}(a)\exp(-k_2\int_{A_-(\mu)}^a\overline{\mu}_-^{\thinspace\frac12}(x)dx),\ a\in(A_-(\mu),0).
\end{equation}

Since $F$ and $f$ are continuous at $a=0$, it follows from \eqref{F+}-\eqref{f-}
that $k_1$ and $k_2$ must satisfy
\begin{equation*}
\begin{aligned}
&\exp(-k_1\int_0^{A_+(\mu)}\overline{\mu}_+^{\thinspace\frac12}(a)da)+
\exp(-k_2\int_{A_-(\mu)}^0\overline{\mu}_-^{\thinspace\frac12}(a)da)=1;\\
&k_1\exp(-k_1\int_0^{A_+(\mu)}\overline{\mu}_+^{\thinspace\frac12}(a)da)=
k_2\exp(-k_2\int_{A_-(\mu)}^0\overline{\mu}_-^{\thinspace\frac12}(a)da),
\end{aligned}
\end{equation*}
or equivalently,
\begin{equation}\label{k1k2}
\begin{aligned}
&\exp(-k_1\int_0^{A_+(\mu)}\overline{\mu}_+^{\thinspace\frac12}(a)da)=\frac{k_2}{k_1+k_2};\\
&\exp(-k_2\int_{A_-(\mu)}^0\overline{\mu}_-^{\thinspace\frac12}(a)da)=\frac{k_1}{k_1+k_2}.
\end{aligned}
\end{equation}
When a  function $F$, satisfying \eqref{F+} and \eqref{F-}, with $k_1,k_2$ satisfying \eqref{k1k2}, is substituted into
the left hand sides of \eqref{key1}, it will render these expressions constant in $a$, since
$F$ has been obtained by solving the differential equation obtained by setting to 0 the
derivative of  the left hand  side in each of the two equations in \eqref{key1}.
However, in order to conclude that such an $F$ is indeed a critical point of $G_2$, we still need to verify that
\eqref{key1} holds. Since the left hand sides are constant in $a$, it suffices to verify
these the equations at $a=0$ (actually as $a\to0^-$ for the first equation and as $a\to0^+$ for the second one). This yields the requirement
\begin{equation}\label{addrequire}
\begin{aligned}
&\frac{1-F(0)}{f^2(0)}=\int_{A_-(\mu)}^0\frac{\overline{\mu}_-(a)}{f(a)}da+\frac p{1-p}\int_0^{A_+(\mu)}\frac{\overline{\mu}_+(a)}{f(a)}da,\\
&\frac{F(0)}{f^2(0)}=\frac{1-p}p\int_{A_-(\mu)}^0\frac{\overline{\mu}_-(a)}{f(a)}da+
\int_0^{A_+(\mu)}\frac{\overline{\mu}_+(a)}{f(a)}da.
\end{aligned}
\end{equation}
We now use \eqref{F+}-\eqref{f-} to write \eqref{addrequire} exclusively in terms of $\mu_+$, $\mu_-$, $k_1,k_2$ and  $p$.
Using  \eqref{F+}-\eqref{f-}, we have
\begin{equation}\label{boundarycond}
\begin{aligned}
&\frac{F(0)}{f^2(0)}=\frac1{k_1^2}\exp(k_1\int_0^{A^+(\mu)}\overline{\mu}_+^{\thinspace\frac12}(a)da),\\
&\frac{1-F(0)}{f^2(0)}=\frac1{k_2^2}\exp(k_2\int_{A^-(\mu)}^0\overline{\mu}_-^{\thinspace\frac12}(a)da).
\end{aligned}
\end{equation}
Using \eqref{f+}, we have
\begin{equation}\label{boundarycondagain+}
\begin{aligned}
&\int_0^{A_+(\mu)}\frac{\overline{\mu}_+(a)}{f(a)}da=
\frac1{k_1}\int_0^{A_+(\mu)}\overline{\mu}_+^{\thinspace\frac12}(a)
\exp(k_1\int_a^{A_+(\mu)}\overline{\mu}_+^{\thinspace\frac12}(x)dx)=\\
&\frac1{k_1^2}\Big(\exp(k_1\int_0^{A_+(\mu)}\overline{\mu}_+^{\thinspace\frac12}(a)da)-1\Big),
\end{aligned}
\end{equation}
and similarly, using \eqref{f-}, we obtain
\begin{equation}\label{boundarycondagain-}
\int_{A_-(\mu)}^0\frac{\overline{\mu}_-(a)}{f(a)}da=
\frac1{k_2^2}\Big(\exp(k_2\int_{A_-(\mu)}^0\overline{\mu}_-^{\thinspace\frac12}(a)da)-1\Big).
\end{equation}
From \eqref{boundarycond}-\eqref{boundarycondagain-}, the requirement in \eqref{addrequire} can be written as
\begin{equation}\label{addrequireagain}
\begin{aligned}
&\frac{1-p}p=\frac{k_2^2}{k_1^2}\Big(\exp(k_1\int_0^{A_+(\mu)}\overline{\mu}_+^{\thinspace\frac12}(a)da)-1\Big),\\
&\frac p{1-p}=\frac{k_1^2}{k_2^2}\Big(\exp(k_2\int_{A_+(\mu)}^0\overline{\mu}_-^{\thinspace\frac12}(a)da)-1\Big).
\end{aligned}
\end{equation}

Thus, we have shown that $F$ is a critical point if and only if it satisfies \eqref{F+} and \eqref{F-}, where
$k_1$ and $k_2$ satisfy
 \eqref{k1k2} and \eqref{addrequireagain}.
 However, the pair $k_1,k_2$ is over-determined by \eqref{k1k2} and \eqref{addrequireagain}.
From these two equations, it follows that
\begin{equation}\label{kp}
\begin{aligned}
&\frac{k_1}{k_1+k_2}=p\thinspace;\\
&\frac{k_2}{k_1+k_2}=1-p.
\end{aligned}
\end{equation}
Substituting this back into \eqref{k1k2} gives
\begin{equation*}
\begin{aligned}
&\exp(-k_1\int_0^{A_+(\mu)}\overline{\mu}_+^{\thinspace\frac12}(a)da)=1-p;\\
&\exp(-k_2\int_{A_-(\mu)}^0\overline{\mu}_-^{\thinspace\frac12}(a)da)=p,
\end{aligned}
\end{equation*}
or equivalently
\begin{equation}\label{finalk}
\begin{aligned}
k_1=\frac{|\log(1-p)|}{\int_0^{A_+(\mu)}\overline{\mu}_+^{\thinspace\frac12}(a)da} \ \ \text{and}\ \
k_2=\frac{|\log p|}{\int_{A_-(\mu)}^0\overline{\mu}_-^{\thinspace\frac12}(a)da}.
\end{aligned}
\end{equation}
But \eqref{kp} and \eqref{finalk} hold simultaneously if and only if $\mu$ satisfies the square root balance
condition \eqref{srbc}.

We have now shown that if $\mu$ satisfies the square root balance condition and if  the restriction of
$\mu$  to $(A_-(\mu),A_+(\mu))$   is absolutely continuous with a piecewise continuous, locally
bounded density on, then $G_2$ possesses a unique critical point, call it  $F_0$, while otherwise $G_2$ has no critical points.
This critical point $F_0$ is given by \eqref{F+} and \eqref{F-}, where $k_1,k_2$ are as in  \eqref{kp}:
\begin{equation}\label{F0}
\begin{aligned}
&F_0(a)=\exp\Big(-\frac{|\log(1-p)|}{\int_0^{A_+(\mu)}\overline{\mu}_+^{\thinspace\frac12}(x)dx}\int_a^{A_+(\mu)}\overline{\mu}_+^{\thinspace\frac12}(x)dx\Big),\ a\in(0,A_+(\mu)),\\
&F_0(a)=1-\exp\Big(-\frac{|\log p|}{\int_{A_-(\mu)}^0\overline{\mu}_-^{\thinspace\frac12}(x)dx}\int_{A_-(\mu)}^a\overline{\mu}_-^{\thinspace\frac12}(x)dx\Big),\ a\in(A_-(\mu),0).
\end{aligned}
\end{equation}
Recall that  distributions $F$ are connected to drifts $b$ via \eqref{def-F}; thus $b=\frac D2\frac{F''}{F'}=\frac D2\frac{f'}{f}$.
Using this with \eqref{F0}, it follows that the drift $b_0$ associated with $F_0$ is given by \eqref{bestb}.

We now show $b_0$ constitutes the unique global minimum of $G_1$. Uniqueness is immediate.
Indeed, if $b_1$ is also the global minimum, then $F_1$ would be a critical point for $G_2$,
where $F_1$ corresponds to $b_1$ via \eqref{def-F}; however $F_0$ is the unique critical point of $G_2$.

We turn to showing that the global minimum occurs at $b_0$.
Recall that we are  assuming that  the restriction of $\mu$  to  $(A_-(\mu),A_+(\mu))$
is absolutely continuous with  piecewise continuous, locally bounded density.
First assume that $\mu$ possesses atoms at $A_-(\mu)$ and $A_+(\mu)$  as in Proposition \ref{G1G2}.
Let $b\in\mathcal{D}_\mu$,  and define $L_1$ as in Proposition \ref{G1G2}. Then it follows from that proposition and Proposition \ref{convex}
that $G_1(b)\ge G_1(b_0)$.

Now assume that $\mu$ possesses an atom at $A_-(\mu)$ but not at $A_+(\mu)$. (The cases in which $\mu$ possesses an atom at
$A_+(\mu)$ but not at $A_-(\mu)$, or in which $\mu$ possesses atoms at both $A_-(\mu)$ and $A_+(\mu)$ are treated similarly.)
For each $n\in\mathbb{N}$,  approximate $\mu_+$ by $\mu_{+;n}$, defined as  follows. If $A_+(\mu)<\infty$, let $\mu_{+,n}$ restricted to $(0,A_+(\mu)-\frac1n)$ coincide
with $\mu_+$ restricted to  $(0,A_+(\mu)-\frac1n)$. Also, let $\mu_{+;n}$ have an atom of mass $\mu_+\big( A_+(\mu)-\frac1n,A_+(\mu) \big)$
at $A_+(\mu)-\frac1n$. If $A_+(\mu)=\infty$,
let $\mu_{+,n}$ restricted to $(0,n)$ coincide
with $\mu_+$ restricted to  $(0,n)$. Also, let $\mu_{+;n}$ have an atom of mass $\mu_+\big((n,\infty)\big)$ at $n$.
Then $\mu_{+;n}$ converges weakly to $\mu_+$, and since the integrands are monotone, we have
\begin{equation}\label{munfrac12}
\lim_{n\to\infty}\int_0^\infty\mu^{\frac12}_{+;n}(x)dx=\int_0^\infty\mu^\frac12_+(x)dx.
\end{equation}
Now define $\mu_n=(1-p_n)\mu_-+p_n\mu_{+;n}$, where $p_n$ is defined so that $\mu_n$ satisfies the square
root balance condition \eqref{srbc}.
Note that according to the previous paragraph, $\mu_n$ is a measure of the type for which
the critical $b$, call it $b_n$, is in fact the global minimum of $G_1$.
It is given by \eqref{bestb}, with $\overline{\mu}_+$  replaced by
$\overline{\mu}_{+;n}$.
Substituting this drift  in \eqref{bform}, or equivalently in \eqref{G1}, and performing the routine calculation gives
\begin{equation}\label{muninf}
\begin{aligned}
&\inf_{b\in\mathcal{D}_{\mu_n}}\int_{\mathbb{R}}(E_0^{(b)}T_a)\mu_n(da)=\\
&\frac2D\Big(\frac{1-p_n}{|\log p_n| }\thinspace(\int_{-\infty}^0\overline{\mu}_-^\frac12(x)dx)^2+\frac{ p_n}{|\log(1-p_n)|}\thinspace(\int_0^\infty\overline{\mu}_{+;n}^\frac12(x)dx)^2\Big).
\end{aligned}
\end{equation}
By \eqref{munfrac12} and the fact that $\mu$ satisfies the square root balance condition, it follows that
\begin{equation}\label{pnconverges}
\lim_{n\to\infty}p_n=p.
\end{equation}
Since $E_0^{(b)}T_a$ is an increasing function of $a\in(0,\infty)$, by the construction of $\mu_n$ we have
\begin{equation}\label{approxmumun}
\int_{\mathbb{R}}(E_0^{(b)}T_0)\mu(da)\ge\int_{\mathbb{R}}(E_0^{(b)}T_0)\mu_n(da),\ \text{for any drift}\ b.
\end{equation}
The  critical  drift $b_0$ that we have found is  given by  \eqref{bestb}.
Substituting this drift  in \eqref{bform}  and performing the routine calculation gives
\begin{equation}\label{infforcritdrift}
\begin{aligned}
&\int_{\mathbb{R}} (E^{(b_0)}_0T_a)\thinspace\mu(da)=\\
&\frac2D\Big(\frac{1-p}{|\log p| }\thinspace(\int_{-\infty}^0\overline{\mu}_-^\frac12(x)dx)^2+\frac p{|\log(1-p)|}\thinspace(\int_0^\infty\overline{\mu}_+^\frac12(x)dx)^2\Big).
\end{aligned}
\end{equation}
From \eqref{munfrac12}-\eqref{infforcritdrift}, we conclude that $b_0$ is indeed the global minimum of $G_1$.

To complete the proof of part (ii), it remains to prove the statements that follow \eqref{bestb}.
The function
 $u(a)=\int_0^adx\exp(-\frac2D\int_0^xb_0(t)dt)$ is harmonic for the diffusion generator $\frac D2\frac{d^2}{dx^2}+b_0(x)\frac{d}{dx}$. Thus, by Ito's formula it follows that
$$
P^{(b_0)}_0(\tau_{a_1}<\tau_{a_2})=\frac{u_0(0)-u_0(a_2)}{u_0(a_1)-u_0(a_2)},\ \text{for}\ A_-(\mu)<a_1<0<a_2<A_+(\mu).
$$
 Substituting for $b_0$ above from \eqref{bestb},
 we have
 $$
 \begin{aligned}
 &u_0(a)=\begin{cases} -\int_a^0\overline{\mu}_-^{\thinspace-\frac12}(x)\exp\big(\frac{|\log p|}{\int_{-\infty}^0\overline{\mu}_-^\frac12(x)dx}\int_0^x\overline{\mu}_-^\frac12(y)dy\big),\ A_-(\mu)<a<0;\\
 \int_0^a\overline{\mu}_+^{\thinspace-\frac12}(x)\exp(-\frac{|\log(1- p)|}{\int_0^{\infty}\overline{\mu}_+^\frac12(x)dx}\int_0^x\overline{\mu}_+^\frac12(y)dy),\ 0<a< A_+(\mu).
  \end{cases}
 \end{aligned}
 $$
Then  $\lim_{a_1\to A_-(\mu)^+}u_0(a_1)$ is infinite or finite depending on whether $\int_{A_-(\mu)}\overline{\mu}^{\thinspace -\frac12}(x)dx$ is infinite or finite.
In the former case, $\lim_{a_1\to  A_-(\mu)^+}P^{(b_0)}_0(\tau_{a_1}<\tau_{a_2})$ is equal to 0, and in the latter case it is positive.
  The exact same argument holds with regard to $\overline{\mu}_+$ and $A_+(\mu)$.
\hfill$\square$
\medskip

\noindent \it Proof of part (i).\rm\
  If $\mu$ is such that the infimum is attained, as specified in part (ii), then substituting the optimal drift from \eqref{bestb} in \eqref{bform}, and performing the routine calculation shows that
 \eqref{infb} holds.

 Now assume that $\mu=(1-p)\mu_-+p\mu_+$ satisfies the square root balance condition and is  such that the infimum is not attained, as specified in part (ii).
For each $n\in\mathbb{N}$, define measures $\mu_{+;n,+}$ and $\mu_{+; n,-}$ on $(0,\infty)$ as follows.
For $k=0,1,\cdots$, let $\mu_{+;n,+}$, when restricted to $(\frac{k+1}n,\frac{k+2}n)$, be uniform with
total mass equal to $\mu_+((\frac kn,\frac{k+1}n])$.
For $k=2,\cdots$, let $\mu_{+;n,-}$, when restricted to $(\frac{k-1}n,\frac kn)$, be uniform with total
mass equal to $\mu_+((\frac kn,\frac{k+1}n])$.
Also, let $\mu_{+;n,-}$, when restricted to $(0,\frac 1n)$, be uniform with total
mass equal to $\mu_+((0,\frac2n])$
  Define respectively $\mu_{-;n,+}$ and $\mu_{-;n,-}$ in a parallel fashion for
$\mu_-$ as $\mu_{+;n,+}$ and $\mu_{+;n,-}$ were defined for $\mu_+$. Then $\mu_{+;n,+}$ and $\mu_{+;n,-}$ both converge weakly to $\mu_+$, and
$\mu_{-;n,+}$ and $\mu_{-;n,-}$ both converge weakly to $\mu_-$, as $n\to\infty$.
Therefore, since all the integrands  are  monotone, we have
\begin{equation}\label{convergence4times}
\begin{aligned}
&\lim_{n\to\infty}\int_0^\infty\overline{\mu}_{+;n,+}^\frac12(x)dx=
\lim_{n\to\infty}\int_0^\infty\overline{\mu}_{+;n,-}^\frac12(x)dx=\int_0^\infty\overline{\mu}_+^\frac12(x)dx,\\
&\lim_{n\to\infty}\int_{-\infty}^0\overline{\mu}_{-;n,+}^\frac12(x)dx=
\lim_{n\to\infty}\int_{-\infty}^0\overline{\mu}_{-;n,-}^\frac12(x)dx=\int_{-\infty}^0\overline{\mu}_-^\frac12(x)dx.
\end{aligned}
\end{equation}

Now define
$$
\begin{aligned}
&\mu_{+,n}=(1-p_{+,n})\mu_{-;n,+}+p_{+,n}\thinspace\mu_{+;n,+}\thinspace;\\
&\mu_{-,n}=(1-p_{-,n})\mu_{-;n,-}+p_{-,n}\thinspace\mu_{+;n,-}\thinspace,
\end{aligned}
$$
where $p_{+,n}$ and $p_{-,n}$ are defined
so that $\mu_{+,n}$ and $\mu_{-,n}$ satisfy the square root balance condition \eqref{srbc}.
By \eqref{convergence4times} and the fact that $\mu$ satisfies the square root balance condition, it follows
that
\begin{equation}\label{pnconv}
\lim_{n\to\infty}p_{+,n}=\lim_{n\to\infty}p_{-,n}=p.
\end{equation}
By part (ii),  the measures $\mu_{+,n}$ and $\mu_{-,n}$ are of the type for which the infimum is attained; let
$b_{+,n}$ and $b_{-,n}$ denote the corresponding drifts for which the infimum is attained.
Then
\begin{equation}\label{infmu+n-n}
\begin{aligned}
&\inf_{b\in \mathcal{D}_{\mu_{+,n}}}\int_{\mathbb{R}} (E^{(b)}_0T_a)\thinspace\mu_{+,n}(da)=
\int_{\mathbb{R}} (E^{(b_{+,n})}_0\thinspace T_a)\thinspace\mu_{+,n}(da)=\\
&\frac2D\Big(\frac{1-p_{+,n}}{|\log p_{+,n}| }\thinspace(\int_{-\infty}^0\overline{\mu}_{-;n,+}^\frac12(x)dx)^2+\frac{ p_{+,n}}{|\log(1-p_{+,n})|}\thinspace(\int_0^\infty\overline{\mu}_{+;n,+}^\frac12(x)dx)^2\Big);\\
&\inf_{b\in \mathcal{D}_{\mu_{-,n}}}\int_{\mathbb{R}} (E^{(b)}_0T_a)\thinspace\mu_{-,n}(da)=
\int_{\mathbb{R}} (E^{(b_{-,n})}_0\thinspace T_a)\thinspace\mu_{-,n}(da)=\\
&\frac2D\Big(\frac{1-p_{-,n}}{|\log p_{-,n}|}\thinspace(\int_{-\infty}^0\overline{\mu}_{-;n,-}^\frac12(x)dx)^2+\frac{ p_{-,n}}{|\log(1-p_{-,n})|}\thinspace(\int_0^\infty\overline{\mu}_{+;n,-}^\frac12(x)dx)^2\Big).
\end{aligned}
\end{equation}
Since $E_0^{(b)}T_a$ is an increasing function of $a\in(0,\infty)$ and a decreasing function
of $a\in(-\infty,0)$, it follows from the construction that
\begin{equation}\label{infsupallb}
\begin{cases}\int_0^\infty (E_0^{(b)}T_a)\mu_+(da)\le\int_0^\infty (E_0^{(b)}T_a)\mu_{+;n,+}(da),\\
\int_{\frac1n}^\infty (E_0^{(b)}T_a)\mu_{+;n,-}(da)\le\int_{\frac2n}^\infty (E_0^{(b)}T_a)\mu_+(da),\\
\int_{-\infty}^0 (E_0^{(b)}T_a)\mu_-(da)\le\int_{-\infty}^0 (E_0^{(b)}T_a)\mu_{-;n,+}(da),\\
\int_{-\infty}^{-\frac1n}(E_0^{(b)}T_a)\mu_{-;n,-}(da)\le\int_{-\infty}^{-\frac2n}(E_0^{(b)}T_a)\mu_-(da),
\end{cases}
\ \text{for any drift}\ b.
\end{equation}
 The proof of part (i) now follows from \eqref{convergence4times}-\eqref{infsupallb}.\hfill $\square$
\medskip

\noindent \it Proof of part (iii).\rm\ In the proof of part (i) above, we proved the statement in part (iii) for
two particular sequences  $\{b_n\}_{n=1}^\infty$; namely for what we called  $\{b_{+,n}\}_{n=1}^\infty$
and $\{b_{-,n}\}_{n=1}^\infty$.
We leave it to the reader to do the routine analysis to show that the result holds more generally as stated in part (iii).
\hfill $\square$
\section{Proofs of Theorems \ref{thm1} and \ref{thm3}}\label{pfthm13}
\it \noindent Proof of Theorem \ref{thm1}.\rm\
For the proof of part (i) of Theorem \ref{thm2}, we constructed the measures
$\mu_{-,n}=(1-p_{-,n})\mu_{-;n,-}+p_{-,n}\thinspace\mu_{+;n,-}$.
For the proof here we consider the measures
\begin{equation}\label{truncatedver}
\mu_{-,n,\text{tr}}:=(1-p_{-,n,\text{tr}})\mu_{-;n,-,\text{tr}}+p_{-,n,\text{tr}}\thinspace\mu_{+;n,-,\text{tr}},
\end{equation}
where
$\mu_{+;n,-,\text{tr}}$ and $\mu_{-;n,-,\text{tr}}$ are appropriately truncated versions of
$\mu_{+;n,-}$ and $\mu_{-;n,-}$, and $p_{-,n,\text{tr}}$ is chosen so that $\mu_{-,n,\text{tr}}$ satisfies
the square root balance condition.

The truncated version, $\mu_{+;n,-,\text{tr}}$   of $\mu_{+;n,-}$, is defined as follows.
Let $\mu_{+;n,-,\text{tr}}$, restricted to
$(0,n]$
 coincide with   $\mu_{+;n,-}$ on  $(0,n]$, and
 let $\mu_{+;n,-,\text{tr}}$, restricted to
$(n,n+\frac1n)$ be uniform with total mass equal to $\mu_+((n+\frac1n,\infty))$.
The truncated version, $\mu_{-;n,-,\text{tr}}$   of $\mu_{-;n,-}$, is defined
in the exact parallel fashion on $(-\infty,0)$.

The measures $\mu_{-;n,-\text{tr}}$ and
$\mu_{+;n,-\text{tr}}$ converge weakly to $\mu_{-;n}$ and to $\mu_{+;n}$; thus,
\begin{equation}\label{infsquareroot}
\begin{aligned}
&\lim_{n\to\infty}\int_{-\infty}^0\overline{\mu}_{-;n,-,\text{tr}}^\frac12(x)dx=
\int_{-\infty}^0\overline{\mu}_-^\frac12(x)dx=\infty;\\
&\lim_{n\to\infty}\int_0^\infty\overline{\mu}_{+;n,-,\text{tr}}^\frac12(x)dx=
\int_0^\infty\overline{\mu}_+^\frac12(x)dx=\infty.
\end{aligned}
\end{equation}
By part (ii)  of Theorem \ref{thm1}, the measure $\mu_{-,n,\text{tr}}$ is  of the type for which the infimum is attained;
let $b_{-,n,\text{tr}}$ denote the corresponding drift  for which the infimum is attained. Then
\begin{equation}\label{infforapprox}
\begin{aligned}
&\inf_{b\in \mathcal{D}_{\mu_{-,n,\text{tr}}}}\int_{\mathbb{R}} (E^{(b)}_0T_a)\thinspace\mu_{-,n,\text{tr}}(da)=
\int_{\mathbb{R}} (E^{(b_{-,n,\text{tr}})}_0\thinspace T_a)\thinspace\mu_{-,n,\text{tr}}(da)=\\
&\frac2D\Big(\frac{1-p_{-,n,\text{tr}}}{|\log p_{-,n,\text{tr}}|}\thinspace(\int_{-\infty}^0\overline{\mu}_{-;n,-,\text{tr}}^\frac12(x)dx)^2+\frac{ p_{-,n,\text{tr}}}{|\log(1-p_{-,n,\text{tr}})|}\thinspace(\int_0^\infty\overline{\mu}_{+;n,-,\text{tr}}^\frac12(x)dx)^2\Big).
\end{aligned}
\end{equation}
As in \eqref{infsupallb}, we have
\begin{equation}\label{infallb}
\begin{cases}
\int_{\frac1n}^\infty (E_0^{(b)}T_a)\mu_{+;n,-,\text{tr}}(da)\le\int_{\frac2n}^\infty (E_0^{(b)}T_a)\mu_+(da),\\
\int_{-\infty}^{-\frac1n}(E_0^{(b)}T_a)\mu_{-;n,-,\text{tr}}(da)\le\int_{-\infty}^{-\frac2n}(E_0^{(b)}T_a)\mu_-(da),
\end{cases}
\ \text{for any drift}\ b.
\end{equation}
In our construction, we have no control over $p_{-,n,\text{tr}}\in(0,1)$.  Note
that
$$
\lim_{p\to1}\frac{1-p}{|\log p|}=\lim_{p\to0}\frac p{|\log(1-p)|}=1;\ \  \lim_{p\to0}\frac{1-p}{|\log p|}=\lim_{p\to1}\frac p{|\log(1-p)|}=0.
$$
Keeping this in mind, Theorem \ref{thm1} now follows from \eqref{infsquareroot}-\eqref{infallb}.
\hfill $\square$
\medskip

\noindent \it Proof of Theorem \ref{thm3}.\rm\ By assumption, $\mu$ does not satisfy
the square root balance condition. The
 proof of part (ii) of Theorem \ref{thm2} revealed that
in such a case, there are no critical points. Thus, the infimum is not attained, proving part (i).
Part (ii) follows from part (i) and part (iii). For part (iii), one substitutes the drift from
\eqref{bestb} into
\eqref{bform} and performs a routine calculation. One obtains the expression in Remark 1 after
the statement of the theorem. A bit of algebra converts this to the expression on the right hand side
of \eqref{notinf}.
\hfill $\square$

\section{Proof of Propostion \ref{exphit}}\label{pfexphit}
We will prove the proposition for $a\in(0,A_+(\mu))$; a similar proof holds for $a\in(A_-(\mu),0)$.
First assume that $A_-(\mu)>-\infty$ and that the diffusion cannot reach $A_-(\mu)$; that is $P_0(T_{A_-(\mu)}<\infty)=0$. As was noted in the discussion following \eqref{support},
this is equivalent to the condition
\begin{equation}\label{recurrencecond}
\int_{A_-(\mu)}dx\exp(-\frac2D\int_0^xb(y)dy)=\infty.
\end{equation}
Then $E_0T_a=\lim_{n\to\infty}E_0T_a\wedge T_{A_-(\mu)+\frac1n}$. Define $u_n(x)=E_xT_a\wedge T_{A_-(\mu)+\frac1n}$, for $x\in[A_-(\mu)+\frac1n,a]$.
By Ito's formula, $u_n$ solves the differential equation
\begin{equation}\label{un}
\begin{aligned}
&\frac D2u_n''+b(x)u_n'=-1,\ x\in(A_-(\mu)+\frac1n,a);\\
&u_n(A_-(\mu+\frac1n))=u_n(a)=0.
\end{aligned}
\end{equation}
Writing the differential equation in the form
$$
\frac D2\Big(\exp\big(\int_a^x\frac 2Db(y)dy\big)u_n'(x)\Big)'=-\exp(\int_a^x\frac 2Db(y)dy),
$$
integrating twice and using the boundary conditions, we obtain
\begin{equation}\label{un}
\begin{aligned}
&u_n(x)=-u_n'(a)\int_a^xdy\exp(-\int_a^y\frac 2Db(t)dt)-\\
&\frac2D\int_x^ady\exp(-\int_a^y\frac 2Db(t)dt)\int_y^a\exp(\int_a^z\frac 2Db(t)dt),
\end{aligned}
\end{equation}
where
$$
u_n'(a)=-\frac{\frac2D\int_{A_-(\mu)+\frac1n}^ady\exp(-\int_a^y\frac 2Db(t)dt)\int_y^adz\exp(\int_a^z\frac 2Db(t)dt)}{\int_{A_-(\mu)+\frac1n}^ady\exp(-\int_a^y\frac 2Db(t)dt)}.
$$
By \eqref{recurrencecond},
\begin{equation}\label{constant}
\lim_{n\to\infty}u_n'(a)=-\frac2D\int_{A_-(\mu)}^adz\exp(\frac2D\int_a^zb(t)dt).
\end{equation}
From \eqref{un} and \eqref{constant} it follows that
$$
E_0T_a=\lim_{n\to\infty}u_n(0)= \frac2D\int_0^adx\exp(-\int_a^x\frac2Db(y)dy)\int_{A_-(\mu)}^xdz\exp(\int_a^z\frac2Db(t)dt).
$$
The number $a$ appearing twice as a lower limit of an integral on the right hand side above can be changed to any other value without changing the value of the right hand side.
Changing $a$ to 0 gives the formula for $E_0T_a$ appearing in the statement of the proposition.

In the case that $A_-(\mu)=-\infty$,  our assumption \eqref{posrec} of positive recurrence ensures that \eqref{recurrencecond} holds, and thus that
$P_0(T_{-\infty}<\infty)=0$. Thus, $E_0T_a=\lim_{n\to\infty}E_0T_a\wedge T_{-n}$.
One now proceeds as in the case  treated above, replacing $A_-(\mu)+\frac1n$ by $-n$.

In the case that $A_-(\mu)>-\infty$ and that \eqref{recurrencecond} does not hold, one has $P_0(T_{A_-(\mu)}<\infty)>0$.
In this case, we have placed a drift equal to $+\infty$ to the left of $A_-(\mu)$, and this is equivalent considering the diffusion with reflection
at $A_-(\mu)$.
Let $u(x)=E_xT_a$. Then by Ito's formula,
$u$ satisfies
\begin{equation}
\begin{aligned}
&\frac D2u''+b(x)u'=-1,\ x\in(A_-(\mu),a);\\
&u'(A_-(\mu))=u(a)=0.
\end{aligned}
\end{equation}
Solving this similarly but more simply than we solved the above equations, we obtain    the formula for $E_0T_a$ appearing in the statement of the proposition.
\hfill $\square$
\section{Proofs of Propositions \ref{convex} and \ref{G1G2}}\label{pfprop34}
\noindent \it Proof of Proposition \ref{convex}.\rm\
We first show that the set $\mathcal{D}_\mu$ is convex.
Let $b,\beta\in\mathcal{D}_\mu$. We need to show that  $(1-t)b+t\beta\in\mathcal{D}_\mu$, for $t\in(0,1)$; that is, that $(1-t)b+t\beta$ satisfies \eqref{driftcond} and \eqref{posrec}.
Now \eqref{driftcond} holds trivially. For \eqref{posrec}, we use H\"older's inequality with $p=\frac1{1-t}$ and $q=\frac1t$ to obtain
\begin{equation*}
\begin{aligned}
&\int_{A_-(\mu)}^{A_+(\mu)} dx\exp(\frac2D\int_0^x\big((1-t)b+t\beta\big)(y)dy)\le\\
& \big(\int_{A_-(\mu)}^{A_+(\mu)} dx\exp(\frac2D\int_0^xb(y)dy)\big)^{1-t}
\big(\int_{A_-(\mu)}^{A_+(\mu)} dx\exp(\frac2D\int_0^x\beta(y)dy)\big)^t<\infty.
\end{aligned}
\end{equation*}

We now prove that $G_1$ is convex.
Recall the definition of $G_1$ from \eqref{G1}.
We will show that
$$
H(b):=(1-p)\int_{A_-(\mu)}^0da\thinspace\overline{\mu}_-(a)\Big[\exp(-\int_0^a\frac2Db(y)dy)\int_a^{A_+(\mu)}dz\exp(\int_0^z\frac2Db(s)ds)\Big]
$$
is  convex. The same proof works for the second term in $G_1$.
We rewrite $H$ as
\begin{equation}\label{H1}
H(b)=(1-p)\int_{A_-(\mu)}^0da\thinspace\overline{\mu}_-(a)\int_a^{A_+(\mu)}dz\exp(\int_a^z\frac2Db(s)ds).
\end{equation}
It follows by the  convexity of the function $e^x$ on all of $\mathbb{R}$ that
\begin{equation}\label{convexexp}
\begin{aligned}
&\exp(\int_a^z\frac2D\big((1-t)b+t\beta\big)(s)ds)\le (1-t)\exp(\int_a^z\frac2Db(s)ds)+\\
&t\exp(\int_a^z\frac2D\beta(s)ds),\ 0\le t\le 1.
\end{aligned}
\end{equation}
Substituting \eqref{convexexp} into \eqref{H1} gives
$H\big((1-t)b+t\beta)\le (1-t)H(b)+tH(\beta)$, for $0\le t\le 1$, proving the convexity.\hfill $\square$
\medskip

\noindent \it Proof of Proposition \ref{G1G2}.\rm\
Define
$$
\begin{aligned}
& \hat F_\epsilon(a)=\int_{A_-(\mu)}^adz\exp(\int_0^z\frac2D\big((1-\epsilon)b_0+\epsilon b\big)(t)dt)\ \ \ \text{and}\ \
 \hat f_\epsilon(a)= \hat F_\epsilon'(a).
\end{aligned}
$$
Recalling  $G_1$ from \eqref{G1}, and recalling that $G_2$ from \eqref{G2} has been defined for positive multiples of distribution functions, we can write
\begin{equation}\label{otherform}
\begin{aligned}
&G_1\big((1-\epsilon)b_0+\epsilon b)\big)=(1-p)\int_{A_-(\mu)}^0\overline{\mu}_-(a)\frac{\hat F_\epsilon(A_+(\mu))- \hat F_\epsilon(a)}{ \hat f_\epsilon(a)}+\\
&p\int_0^{A_+(\mu)}\overline{\mu}_-(a)\frac{ \hat F_\epsilon(a)}{\hat f_\epsilon(a)}=G_2(\hat F_\epsilon).
\end{aligned}
\end{equation}
Also define
\begin{equation}\label{hatQ}
\begin{aligned}
&\hat Q(a)=\lim_{\epsilon\to0^+}\frac{ \hat F_\epsilon(a)- \hat F_0(a)}\epsilon=\int_{A_-(\mu)}^adz\big(\int_0^z\frac2D(b-b_0)(t)dt\big)\exp(\int_0^z\frac2Db_0(t)dt);\\
&\hat q(a)=\hat Q'(a).
\end{aligned}
\end{equation}
The second equality in the first line of \eqref{hatQ} follows
from the bounded convergence theorem and the assumptions in the statement of the proposition.
Using \eqref{otherform} and \eqref{hatQ} we have, similar to \eqref{var-1},
\begin{equation}\label{anotherformderiv}
\begin{aligned}
&L'(0^+)=\lim_{\epsilon\to0^+}\frac{L_1(\epsilon)-L_1(0)}\epsilon=\\
&\lim_{\epsilon\to0^+}\frac{G\big((1-\epsilon)b_0+\epsilon b)-G_1(b_0)}\epsilon=
\lim_{\epsilon\to0^+}\frac{G_2(\hat F_\epsilon)-G_2(\hat F_0)}\epsilon=\\
&(1-p)\int_{A_-(\mu)}^0\overline{\mu}_-(a)\Big(\frac{\hat Q(A_+(\mu))-\hat Q(a)}{ f_0(a)}-\frac{(1-F_0(a))\hat q(a)}{f_0^2(a)}\Big)da+\\
&p\int_0^{A_+(\mu)}\overline{\mu}_+(a)\Big(\frac{\hat Q(a)}{f_0(a)}-\frac{F_0(a)\hat q(a)}{f_0^2(a)}\Big)da.
\end{aligned}
\end{equation}
The second equality above follows from the bounded convergence theorem and the assumptions in the statement of the proposition.

We need to consider the cases  $\hat Q(A_+(\mu))\neq0$ and    $\hat Q(A_+(\mu))=0$ separately. First consider the case
  $\hat Q(A_+(\mu))\neq0$. Define $\bar Q(a)=\frac{\hat Q(a)}{\hat Q(A_+(\mu))}$ and $\bar q(a)=\bar Q'(a)$.
Then the right hand side of \eqref{anotherformderiv} will be equal to 0 if and only if
\begin{equation}\label{likevar-1}
\begin{aligned}
&(1-p)\int_{A_-(\mu)}^0\overline{\mu}_-(a)\Big(\frac{1-\bar Q(a)}{ f_0(a)}-\frac{(1-F_0(a))\bar q(a)}{f_0^2(a)}\Big)da+\\
&p\int_0^{A_+(\mu)}\overline{\mu}_+(a)\Big(\frac{\bar Q(a)}{f_0(a)}-\frac{F_0(a)\bar q(a)}{f_0^2(a)}\Big)da=0.
\end{aligned}
\end{equation}
Recall that since $F_0$ is the critical point of $G_2$, \eqref{var-1} holds with $F_0$ and $f_0$ substituted for $F$ and $f$.
The only difference between \eqref{var-1}, with $F_0$ and $f_0$ substituted for $F$ and $f$, and \eqref{likevar-1} is that $\hat  Q$ and $\hat q$ appear  in \eqref{likevar-1}  while $Q$ and $q$ appear in
 \eqref{var-1}, where  $Q$ is a distribution function with compactly supported density $q$. However, the  analysis from \eqref{var-1} to \eqref{key1} goes through just the same for $\bar Q$, since $\bar Q(A_+(\mu))=1$. (Neither the monotonicity of $Q$ nor the compact support of $q$ was used there; only the fact
that $Q(A_+(\mu))=1$.)
Thus, the right hand side of \eqref{anotherformderiv} is equal to 0, proving the proposition.

Now consider the case   $\hat Q(A_+(\mu))=0$. Because $\hat Q(A_+(\mu))=0$, whereas in \eqref{var-1} one had $Q(A_+(\mu))=1$,  the analysis that showed that the left hand side of \eqref{var-1} is equal to the left hand side of \eqref{var-2}, when applied to the last two lines of  \eqref{anotherformderiv},
shows that
\begin{equation}\label{QA0}
\begin{aligned}
&(1-p)\int_{A_-(\mu)}^0\overline{\mu}_-(a)\Big(\frac{\hat Q(A_+(\mu))-\hat Q(a)}{ f_0(a)}-\frac{(1-F_0(a))\hat q(a)}{f_0^2(a)}\Big)da+\\
&p\int_0^{A_+(\mu)}\overline{\mu}_+(a)\Big(\frac{\hat Q(a)}{f_0(a)}-\frac{F_0(a)\hat q(a)}{f_0^2(a)}\Big)da=\\
&(1-p)\int_{A_-(\mu)}^0\hat q(a)\Big[-\int_a^0\frac{\overline{\mu}_-(x)}{f_0(x)}dx-\frac{\overline{\mu}_-(a)(1-F_0(a))}{f_0^2(a)}\Big]da+\\
&p\int_0^{A_+(\mu)}\hat q(a)\Big[-\int_0^a\frac{\overline{\mu}_+(x)}{f_0(x)}dx-\frac{\overline{\mu}_+(a)F_0(a)}{f_0^2(a)}\Big].
\end{aligned}
\end{equation}
(The right hand side of \eqref{QA0}  corresponds to the first two lines  of \eqref{var-2}. The two terms on the third line of \eqref{var-2} do not appear now because $\hat Q(A_+(\mu))=0$.)
Because $F_0$ is critical, it follows from \eqref{key1} that
 \begin{equation}\label{C1}
\begin{aligned}
& (1-p)\Big[\int_a^0\frac{\overline{\mu}_-(x)}{f_0(x)}dx+\frac{\overline{\mu}_-(a)(1-F_0(a))}{f_0^2(a)}\Big]=C_1,\ a\in(A_-(\mu),0);\\
&p\Big[-\int_0^a\frac{\overline{\mu}_+(x)}{f_0(x)}dx-\frac{\overline{\mu}_+(a)F_0(a)}{f_0^2(a)}\Big]=C_1,\ \ a\in(0,A_+(\mu)),
\end{aligned}
\end{equation}
where $C_1=(1-p)\int_{A_-(\mu)}^0\frac{\overline{\mu}_-(a)}{f_0(a)}da+p\int_0^{A_+(\mu)}\frac{\overline{\mu}_+(a)}{f_0(a)}da$.
Substituting \eqref{C1} into the right hand side of \eqref{QA0}, we conclude that the right hand side of \eqref{QA0} is equal
to $C_1\int_{A_-(\mu)}^0\hat q(a)da+C_1\int_0^{A_+(\mu)}\hat q(a)da=C_1\int_{A_-(\mu)}^{A_+(\mu)}\hat q(a)da=C_1\hat Q(A_+(\mu))=0$.
Thus, the left hand side of  \eqref{QA0}, which is the right hand side  of \eqref{anotherformderiv}, is equal to 0. \hfill $\square$

\end{document}